\documentclass[11pt]{amsart}
\usepackage{enumitem}
\usepackage{amsmath}
\usepackage{amssymb, amsmath, amscd, amsthm}
\usepackage{graphicx}
\usepackage[sortcites=true,backend=biber,style=numeric-comp,maxbibnames=10]{biblatex}
\usepackage[margin=1in]{geometry}
\usepackage{xcolor}
\usepackage[hidelinks]{hyperref}
\usepackage[final]{showlabels}

\usepackage{pgfplots}
\usepackage{tikz}
\usepackage{tikz-cd}
\tikzcdset{
arrow style=tikz,
diagrams={>=Stealth}}
\usepgflibrary{arrows}
\usepgflibrary{arrows.meta}
\usetikzlibrary{calc}
\usetikzlibrary{decorations.markings}
\usepgfplotslibrary{external}
\tikzifexternalizing{%
  \geometry{driver=none}
}{}
\usetikzlibrary{positioning}

\tikzset{mid arrow head/.style={
    decoration={markings, mark=at position 0.55 with
      {\arrow[scale=1.3]{>}}}, postaction={decorate}
  }
}

\tikzset{mid arrow head reversed/.style={
    decoration={markings, mark=at position 0.55 with
      {\arrow[scale=1.3]{<}}}, postaction={decorate}
  }
}

\tikzstyle{filled}=[
circle,
draw,
solid,
fill=black,
inner sep=2pt,
minimum width=0pt
]

\tikzstyle{vertex}=[
circle,
draw,
solid,
fill=black,
inner sep=0pt,
minimum width=0pt
]

\tikzstyle{triangle label}=[
text=black
]
\tikzstyle{edge label}=[
text=black
]
\tikzstyle{node label}=[black]

\colorlet{TriangleLabel}{black}
\colorlet{EdgeLabel}{black}
\colorlet{NodeLabel}{black}

\tikzset{
  edge node/.code={%
      \expandafter\def\expandafter\tikz@tonodes\expandafter{\tikz@tonodes
        #1}}}
\makeatother
\tikzset{
  subseteq/.style={
    draw=none,
    edge node={node [sloped, allow upside down, auto=false]{$\subset$}}},
  Subseteq/.style={
    draw=none,
    every to/.append style={
      edge node={node [sloped, allow upside down, auto=false]{$\subset$}}}
  }
}

\pgfplotsset{compat=1.18}

\DeclareMathOperator{\sgn}{sgn}

\DeclareMathOperator{\Id}{Id}

\newcommand{\defeq}{\mathrel{\mathop:}=}

\newcommand{\End}{{\rm End}}
\newcommand{\Hom}{{\rm Hom}}
\newcommand{\eval}[2]{\langle #1,\;#2 \rangle}
\newcommand{\Eval}[2]{\bigl\langle #1,\;#2 \bigr\rangle}

\newcommand{\simplex}[2][0]{[#1\,...\,#2]}
\newcommand{\face}[3][\hat{i}]{[#2\,...\,#1\,...\,#3]}
\def\R{\mathbb{R}}
\def\X{\tilde{X}}
\def\Y{\tilde{Y}}
\def\E{\tilde{E}}
\def\tnabla{{\widetilde{\nabla}}}
\def\C{\tilde{C}}
\def\tsigma{\tilde{\sigma}}
\def\coarse{\mathfrak{C}}
\def\Q{\mathcal{Q}}
\def\GL{{\rm GL}}
\def\complex{\mathbb{C}}
\DeclareMathOperator{\sd}{sd}

\newcommand{\SiCo}{\textrm{SiCo}}
\newcommand{\OSiCo}{\textrm{OSiCo}}
\newcommand{\somethingdef}[4]{{\left#1 {#2} \ \left| \ {#3} \vphantom{#2} \right. \right#4}}
\newcommand{\setdef}[2]{\somethingdef{\{}{#1}{#2}{\}}}

\newtheorem{theorem}{Theorem}[section]
\newtheorem{proposition}[theorem]{Proposition}
\newtheorem{corollary}[theorem]{Corollary}
\newtheorem{lemma}[theorem]{Lemma}
\theoremstyle{definition}
\newtheorem{definition}[theorem]{Definition}
\newtheorem{example}[theorem]{Example}
\theoremstyle{remark}

\newtheorem{remark}[theorem]{Remark}

\newtheorem{Construction}[theorem]{Construction}

\graphicspath{{./figs/}}

\begin{document}
\title[Discrete Vector Bundles]{Discrete Vector Bundles with Connection}
\author{Daniel Berwick-Evans}
\email{danbe@illinois.edu}
\address{Department of Mathematics, University of Illinois at Urbana-Champaign, 1409 West Green Street, Urbana, IL 61801}
\author{Anil N. Hirani}
\email{hirani@illinois.edu}
\address{Department of Mathematics, University of Illinois at Urbana-Champaign, 1409 West Green Street, Urbana, IL 61801}
\author{Mark D. Schubel}
\email{mdschubel@gmail.com}
\address{Department of Physics, University of Illinois at Urbana-Champaign, 1110 West Green Street, Urbana, IL 61801}
\curraddr{Apple Inc., One Apple Park Way, Cupertino, CA 95014}
\date{\today}

\begin{abstract}
  We develop a combinatorial theory of vector bundles with connection  on locally ordered simplicial complexes. This is a first step towards a discrete exterior calculus for bundle-valued forms. The basic building block is the discrete exterior covariant derivative, a forward-difference operator defined on bundle-valued cochains. Many standard objects in differential geometry (e.g., curvature, connection 1-forms, gauge transformations) can be understood via the discrete covariant derivative operator,  with their defining formulas identical to the smooth setting. These discrete objects satisfy all of the expected algebraic identities, such as naturality with respect to simplicial maps, and a Bianchi identity for discrete curvature. We also show that flat discrete connections determine a cochain complex that computes twisted de~Rham cohomology in a local coefficient system determined by the discrete vector bundle, with twisted Poincar\'e duality (of densities) being one application. Finally, a coarsening operation applied to bundle-valued cochains provides a direct and concrete comparison with the recent framework for discrete bundles of Christiansen and Hu. 
\end{abstract}

\maketitle

\section{Introduction}

Numerical approaches to partial differential equations often rely on combinatorial models for differential geometric objects. These models can encode structures of interest from a purely mathematical point of view, particularly when the combinatorial model faithfully encodes algebraic identities from geometry. In turn, such algebraic identities ensure that basic geometric concepts are correctly captured by the resulting numerical method. 

A fundamental example comes from the algebraic identities relating the gradient, curl, and divergence. These identities can be phrased (and generalized) in terms of the exterior derivative and Hodge star operators acting on differential forms. There are robust combinatorial models for the Hodge-de~Rham complex in which the standard algebraic identities hold, e.g.,~\cite{Dodziuk1974,ArFaWi2006,ArFaWi2010,Hirani2003,Wilson2007}.
These combinatorial models and their recent generalizations have met success in a wide variety of applications, e.g., computational electromagnetism~\cite{Bossavit1988a,Hiptmair1999}, elasticity~\cite{ArWi2002,ArFaWi2007, Li2018}, numerical relativity~\cite{Quenneville-Belair2015,Li2018}, fluid mechanics~\cite{HiNaCh2015,MoHiSa2016,PaGe2017,NiTeVo2017,JaAbMoSa2021,WaJaHiSa2023}, quantum electrodynamics~\cite{PhFaScTu2023}, photonics~\cite{MoRa2023} and other areas of physics, geometry, and computer graphics. The combinatorial wedge product leads to even more sophisticated algebraic structures, e.g., the failure of associativity is encoded by an $A_\infty$-structure~\cite{DoMoSh2007b,DoMoSh2008,Liu2026}. 

This paper initiates a generalization of these combinatorial models to differential forms with values in a vector bundle. Specifically, we construct combinatorial versions of the exterior covariant derivative $d_\nabla$ and various products. In smooth geometry,  $d_\nabla\circ d_\nabla$ encodes the curvature of the connection. We define a combinatorial version of curvature via this formula, and then show that the resulting structure satisfies all the expected properties, including a Bianchi identity.

One eventual goal is development of coordinate-independent numerical methods for problems involving bundle-valued forms, e.g., elasticity, Yang--Mills equation (see Remark~\ref{rmk:YM}), granular defect dynamics in materials and other. These problems can also be studied in combination (such as magnetohydrodynamics in curved spacetime) necessitating a uniform framework that encompasses all the desired applications. These longer-term goals require enhancements of the theory presented below, e.g., a suitable generalization of the combinatorial Hodge star operator of~\cite{Hirani2003}. 

The goal of the present paper is to provide a basic framework that is both consistent with the latest developments in numerical analysis as well as the direct antecedent of this work, discrete exterior calculus (DEC)~\cite{Hirani2003}. The spaces of cochains and $d_\nabla$ provide a combinatorial model for each row of the Bernstein-Gelfand-Gelfand (BGG) construction~\cite{CaSlSo2001} as used in numerical analysis~\cite{ArHu2021,BeGa2025,HuLi2025}. The interpretation of the ideas of this paper in terms of double forms~\cite{deRham1984} and a cochain discretization of first Bianchi sums needed in BGG~\cite{Kulkarni1972} remain open, although some recent progress has been made~\cite{Zhu2026}.

We emphasize that we do not provide a framework for discretizing smooth bundle-valued forms although Construction~\ref{construction2} sketches one possibility. The discretization of forms in~\cite{BrToGaDe2024} carries this out effectively. Our starting point is that the discrete forms are given.

\subsection*{Review of vector bundle valued forms in smooth geometry}
For $E\to M$ a smooth vector bundle over a smooth manifold $M$, the vector space of $E$-valued differential $k$-forms is defined as~$\Omega^k(M;E) := \Gamma(\Lambda^kT^\ast M\otimes E)$. The graded vector space $\Omega^\bullet(M;E)\cong \bigoplus_k \Omega^k(M;E)$ is a module over the graded algebra $\Omega^\bullet(M)$ of differential forms via the $C^\infty(M)$-linear map 
\begin{equation}
 \Omega^k(M;E)\times \Omega^l(M)\to \Omega^{k+l}(M;E),\qquad 
(\alpha,w)\mapsto  \alpha \wedge w,\quad k,l\in \mathbb{N}.\label{eq:moduleact}
\end{equation}
A connection $\nabla\colon \Gamma(E)\to \Omega^1(M;E)$ on $E$ determines the \emph{exterior covariant derivative} 
\begin{eqnarray}
  \Gamma(E)\simeq\Omega^0(M;E)\stackrel{\nabla=d_\nabla}{\longrightarrow}
  \Omega^1(M;E)\stackrel{d_\nabla}{\to}
  \Omega^2(M;E)\stackrel{d_\nabla}{\to} \Omega^3(M;E)\to
  \cdots \label{eq:sequence}
\end{eqnarray}
This operator $d_\nabla$ is compatible with the module structure~\eqref{eq:moduleact} by way of the Leibniz rule,
\begin{eqnarray}
d_\nabla(\alpha \wedge w)=d_\nabla \alpha \wedge w +(-1)^{|\alpha|}\alpha \wedge d w \label{Leibniz}
\end{eqnarray}
where~$d\colon \Omega^k(M)\to \Omega^{k+1}(M)$ is the ordinary exterior derivative. When $E$ is the trivial line bundle and $\nabla$ is the trivial connection, $d_\nabla=d$ and~\eqref{Leibniz} becomes the familiar Leibniz rule for the exterior derivative. In general,~\eqref{eq:sequence} need not be a cochain complex due to the presence of \emph{curvature} 
\begin{equation}\label{eq:dnabla2}
  F_\nabla:=d_\nabla\circ d_\nabla\in \Omega^2(M;\End(E))
\end{equation} 
defined as a 2-form valued in the endomorphism bundle $\End(E)$ of $E$. Curvature determines an operator on $E$-valued forms via $\Omega^\bullet(M)$-linear maps
\begin{eqnarray}
F_\nabla \colon \Omega^k(M;E)\to \Omega^{k+2}(M;E).\label{eq:actofend}
\end{eqnarray} 
The structures~\eqref{eq:moduleact}-\eqref{eq:actofend} are natural: a smooth map $f\colon N\to M$ determines a pullback bundle $f^\ast E\to N$, and the resulting maps $f^\ast\colon \Omega^\bullet(M;E)\to \Omega^\bullet(N;f^\ast E)$ commute with the exterior covariant derivatives for the pullback connection $f^\ast\nabla$ on $f^\ast E$. In particular, the curvature of $\nabla$ on $E$ pulls back to the curvature of $f^*\nabla$ on $f^*E$. This generalizes naturality of the de~Rham complex.

Another common approach to connections works relative to a trivialization of~$E$ over an open subset~$U\subset M$: for a chosen isomorphism of vector bundles $E|_U\simeq U\times \R^r$, one obtains local expressions for the connection and curvature 
\begin{eqnarray}
d_\nabla\stackrel{{\rm locally}}{=}d+A, \quad F|_U\stackrel{{\rm locally}}{=}F_A:=dA+A\wedge A, \qquad A\in \Omega^1(U;\End(\R^r))\label{eq:smoothA}
\end{eqnarray}
in term of endomorphism valued 1-forms~$A$. A change of the local trivialization of $E|_U$ by $g\colon U\to \GL_r(\R)$ is a \emph{gauge transformation}. This changes the 1-form $A$ by
\begin{eqnarray}
A\mapsto gAg^{-1}-dgg^{-1}.\label{eq:gaugetransf}
\end{eqnarray}
The curvature~\eqref{eq:smoothA} is gauge covariant: $F_A\mapsto gF_Ag^{-1}$. The theory of connections can be rephrased in terms of 1-forms $A$ and the action by gauge transformations~\eqref{eq:gaugetransf}, as is more common in physics.

\begin{remark} \label{rem:taylor}
We recall the sense in which curvature is the infinitesimal expression of parallel transport. Fix a smooth vector bundle $E\to M$ with connection $\nabla$, and local chart $\varphi\colon \R^n\to M$ on which $E$ trivializes. Let $[0,\epsilon]\times [0,\epsilon]\subset \R^n$ be the standard square in the $x_i$-$x_j$-plane with side length $\epsilon$. Fixing a trivialization of $\varphi^*E|_{[0,\epsilon]^2}\simeq [0,\epsilon]\times [0,\epsilon]\times \R^r$ over $[0,\epsilon]\times [0,\epsilon]$, one can integrate the endomorphism valued 2-form $\varphi^*F$, finding
\begin{eqnarray}\label{eq:Taylor}
\int_{[0,\epsilon]^2} \varphi^*F=\int_{[0,\epsilon]^2} F_{ij}dx_i\wedge dx_j={\rm Id}-{\rm Hol}+O(\epsilon^3)\in \End(\R^r)
\end{eqnarray}
where $F_{ij}$ is the $dx_i\wedge dx_j$-component of $F\in \Omega^2(M;\End(E))$ in the chosen chart and ${\rm Hol}$ is the $r\times r$-matrix gotten from the holonomy (i.e., parallel transport) along the boundary of $[0,\epsilon]\times [0,\epsilon]$ based at $(0,0)$ and using the positive orientation on the $x_i$-$x_j$-plane. One method for verifying~\eqref{eq:Taylor} is to express holonomy as a path-ordered exponential and Taylor expand in powers of $\epsilon$, e.g., see \cite[page 247]{BaMu1994}. Our discretization method recovers discrete curvature the approximatation to curvature on the right hand side of~\eqref{eq:Taylor}, see Remark~\ref{rmk:holvscuvature}. 
\end{remark}

\begin{remark}\label{rmk:YM}
One source of vector bundles comes from the associated bundle construction for a principal $G$-bundle $P\to X$ and $G$-representation $\rho\colon G\to \GL(V)$,
\begin{equation}\label{eq:associatedbundle}
E:=P\times_G V.
\end{equation}
Taking the frame bundle of a vector bundle as a $\GL(n)$-principal bundle, all vector bundles arise as examples of~\eqref{eq:associatedbundle}.
For the adjoint representation of $G$ on its Lie algebra $\mathfrak{g}$, define ${\rm Ad}(P):=P\times_G \mathfrak{g}$. The curvature of a principal $G$-bundle with connection is an ${\rm Ad}(P)$-valued 2-form which can can be identified with the square of a covariant derivative operator $d_\nabla\colon \Omega^\bullet(X;{\rm Ad}(P))\to \Omega^\bullet(X;{\rm Ad}(P))$ for the associated connection $\nabla$ under~\eqref{eq:associatedbundle}. In this way, the discrete theory developed below encompasses curvature of principal $G$-bundles, and hence some rudiments of Yang--Mills theory (though we do not address the crucial structure of the Hodge star operator below). 
\end{remark}

\subsection*{Flat bundles and twisted cohomology}
A \emph{flat} connection has vanishing curvature, or equivalently $d_\nabla\circ d_\nabla=0$. This property makes~\eqref{eq:sequence} into a cochain complex that computes the \emph{$E$-twisted cohomology} of $M$, also called \emph{cohomology with local coefficients} (in $E$), denoted ${
\rm H}^\bullet(M;E)$. Twisted cohomology does not have a product, but instead is a module over ordinary cohomology, using the module structure on forms~\eqref{eq:moduleact} and the Leibniz property~\eqref{Leibniz}. 

\begin{remark} 
The covariantly constant sections of a flat vector bundle $(E,\nabla)\to M$ define a locally constant sheaf $\mathcal{E}$, see~\eqref{eq:smoothsheaf} below. The \v{C}ech--de~Rham double complex allows one to identify $E$-twisted de~Rham cohomology of $M$ with sheaf cohomology valued in $\mathcal{E}$. Our discrete framework affords a cochain complex that is \emph{equal} to the \v{C}ech complex for a certain open cover (see the proof of Theorem~\ref{thm:Cech}), thereby providing a combinatorial model for $E$-twisted de~Rham cohomology. 
\end{remark}

\begin{example}[Twisted Poincar\'e duality]\label{ex:twistedPoincare}
An important example of twisted cohomology takes local coefficients in the orientation line, ${\rm or}(M):=\Lambda^n(TM)$ for $n={\rm dim}$: a section of ${\rm or}(M)$ at $p\in M$ is an orientation of $T_pM$. The ${\rm or}(M)$-twisted differential forms in this case are \emph{densities}, which can be integrated independent (and in the absence) of an orientation on $M$. Even in the presence of orientations densities can be useful: integrals of densities are invariant under orientation-reversing diffeomorphisms, which is useful in applications. Combining the integration of densities with the module structure for twisted cohomology affords a map
\begin{eqnarray}
{\rm H}^{n-k}(M;{\rm or}(M))\otimes {\rm H}^k(M)\to {\rm H}^n(M;{\rm or}(M))\xrightarrow{\int} \R.\label{eq:twistedPD}
\end{eqnarray}
\emph{Twisted Poincar\'e duality} is the statement that~\eqref{eq:twistedPD} is a perfect pairing, resulting in an isomorphism between the vector space ${\rm H}^k(M)$ and (the dual of) ${\rm H}^{n-k}(M;{\rm or}(M))$. Twisted Poincar\'e duality is a crucial feature in the geometry of densities, and we emphasize that the existence of the first arrow in~\eqref{eq:twistedPD} relies crucially on the Leibniz rule~\eqref{Leibniz} for $d_\nabla$. 
\end{example}

\subsection*{Statement of results}
This paper develops a theory of discrete vector bundles with connection over simplicial complexes with properties mirroring~\eqref{eq:moduleact}-\eqref{eq:actofend} that are appropriately natural with respect to maps of simplicial complexes. The theory has three anchor points. The first is that when applied to a trivial bundle with trivial connection, we recover the metric-free part of DEC~\cite{Hirani2003} where naturality is the usual pullback of (discrete) cochains~\cite{ScBeHi2024}. The second is that our discrete curvature has geometric interpretation in terms of holonomy (or parallel transport) as in~\eqref{eq:Taylor}. Third and finally, a flat discrete connection affords a cochain complex that computes the expected twisted cohomology, and (in particular) we recover a statement of discrete twisted Poincar\'e duality as in~\eqref{eq:twistedPD}. One additional (and important) upshot is a coarsening procedure that connects our framework with the recent work of Christiansen and Hu~\cite{ChHu2023}; see~\S\ref{sec:coarse}.

To set up the statements of our main theorems, a \emph{discrete vector bundle with connection} on a simplicial complex $X$ is the data of a vector space $E_x$ for each vertex of $X$ and the data of a vector space isomorphism $E_x\xrightarrow{\sim} E_y$ for each edge in~$X$.  Given a discrete vector bundle with connection and $k\in \mathbb{N}$, we construct the vector space $C^k(X;E)$ of \emph{$E$-valued $k$-cochains}. Our first result endows spaces of cochains with the expected structures from differential geometry.

\begin{theorem}[Structure-preserving discretization]\label{mainthm1}
  Let $X$ and $Y$ be locally ordered simplicial complexes,~$E$ a discrete vector bundle with connection over $X$, and $f: Y \to X$ an abstract simplicial map. We construct operations
  \begin{align*}
    d_\nabla &: C^k(X; E) \to C^{k+1}(X; E)\\
    \smile &: C^k(X;E)\times C^l(X) \to C^{k+l}(X; E)\\
    \smile &: C^k(X; \Hom(E)) \times C^l(X;E) \to C^{k+l}(X; E)\\
    d_{\nabla^\Hom} &: C^k(X; \Hom(E)) \to C^{k+1}(X; \Hom(E))\\
    F &:=d_\nabla\circ d_\nabla \in C^2(X; \Hom(E))\\
    f^\ast&: C^k(X;E)\to C^k(Y;f^*E)
  \end{align*}
  such that for $\alpha \in C^k(X; E)$ and $w \in C^l(X)$: 
  \begin{enumerate}[label = (\roman*)]
  \item the Leibniz rule holds (Proposition~\ref{prop:Leibniz})

    \[
      d_\nabla (\alpha \smile w) = d_\nabla \alpha \smile w +
      (-1)^k \alpha\smile d w\, ,
    \]
  \item the operations are natural (Propositions~\ref{prop:cup-naturality} and~\ref{prop:dnaturality})
    \[
    f^\ast(\alpha \smile w) = f^\ast\alpha \smile f^\ast w,\qquad f^\ast d_\nabla = d_{f^\ast\nabla} f^\ast\, ,
    \]
  \item the square of the exterior covariant derivative is the curvature operator (Proposition~\ref{prop:dnabla2-alpha}) \[d_\nabla d_\nabla \alpha = F \smile\alpha,\]
  \item curvature satisfies the Bianchi identity (Proposition~\ref{prop:bianchi}) $d_{\nabla^\Hom}\, F = 0$.
  \end{enumerate}
\end{theorem}

\begin{remark}
When applied to a discrete real line bundle with trivial connection, these statements immediately reduce to previously known results in DEC~\cite{Hirani2003,ScBeHi2024}. 
\end{remark}

The formula~\eqref{eq:smoothA} for $d_\nabla$ in a local trivialization in smooth geometry provides a natural discrete analog of the connection 1-form $A$, namely we take the difference of discrete covariant derivatives
\begin{eqnarray}
A:=d_\nabla-d\label{eq:discreteA}
\end{eqnarray}
where above $d$ is the covariant derivative for the trivial discrete connection. Our second result shows that the discrete connection 1-form~\eqref{eq:discreteA} has expected properties relative to the structures in Theorem~\ref{mainthm1}. 

\begin{theorem}\label{mainthm2}
The discrete 1-form~\eqref{eq:discreteA} satisfies
\begin{enumerate}
    \item the covariant derivative is determined by $A$ as $d_\nabla=d+A$;
    \item the connection 1-form $A$ pulls back to a connection 1-form $f^*A$;
    \item the curvature satisfies $F=dA+A\smile A$ and the Bianchi identity $(d+A)F=0$ holds;
    \item under a gauge transformation $g$ of the discrete vector bundle with connection, $A$ transforms as $A\mapsto gAg^{-1}-dgg^{-1}$, and the curvature transforms as $F\mapsto gFg^{-1}$. 
\end{enumerate}
\end{theorem}

\begin{remark}
Property (1) is a tautology, but confirms that the data of discrete connections and curvature can be reformulated in terms of discrete 1-forms~$A$. Properties (2)-(4) reformulate (ii)-(iv) in Theorem~\ref{mainthm1}. 
\end{remark}

A discrete vector bundle with connection is \emph{flat} if its curvature vanishes, i.e., $F=0$. Such flat bundles provide a bridge between discrete geometry and cohomology with twisted coefficients. 

\begin{theorem}\label{thm:Cech}
Let $(E,\nabla)$ be a discrete bundle with flat discrete connection. The cochain complex $(C^\bullet(X;E),d_\nabla)$ computes the twisted cohomology of $E$ in a local system determined by $(E,\nabla)$. When $(E,\nabla)$ arises as the discretization of a smooth vector bundle with flat connection on a manifold $M$ with triangulation $X$, this twisted cohomology agrees with the twisted de~Rham cohomology of $M$. 
\end{theorem}

Applying the above to the special case of the orientation line ${\rm or}(TM):=\Lambda^{\rm top}(TM)$ of a smooth manifold $M$, we obtain a discrete analog of Example~\ref{ex:twistedPoincare}, see Example~\ref{ex:twistedPoincarediscrete} below. 

\begin{corollary}
The cochain complex of discrete cochains valued in the flat discrete vector bundle associated to the orientation line bundle on $M$ satisfy twisted Poincare duality. 
\end{corollary}

\begin{remark}
In \v{C}ech cohomology, the open cover $\{U_i\}_{i \in I}$ relative to which one defines the \v{C}ech complex is indexed by an \emph{ordered} set (e.g., \cite[pages 92 and 110]{BoTu1982}). This ordering is used to fix signs in the definition of the \v{C}ech differential. Analogously, our discrete theory depends on a (local) ordering of the simplicial complex $X$. As usual, different choices of ordering result in quasi-isomorphic chain complexes. There are tricks for inducing a local ordering on an unordered simplicial complex, e.g., by taking a barycentric subdivision; see~\S\ref{subsec:subdiv} below. 
\end{remark}

A coarsening procedure detailed in~\S\ref{sec:coarse} recovers the discrete vector bundle constructions of Christiansen and Hu~\cite{ChHu2023} for unordered simplicial complexes within our framework.

\subsection{Related work}

Just as there are multiple approaches to the combinatorial de~Rham complex, there are multiple approaches to its vector bundle-valued generalizations. In all cases, one combines structures in a preferred combinatorial model of the de~Rham complex with ideas from lattice gauge theory following Wilson~\cite{Wilson1974} (see also~\cite{MoMu1997}). Our approach generalizes parts of DEC, a combinatorial model for the de~Rham complex that encodes algebraic identities and also can be directly compared to finite element methods~\cite{Zhu2026,ZhChHuHi2025,GuPo2025}; see~\S\ref{sec:DEC} below. With this direct comparison to established numerical methods, we expect the theory developed below to be amenable to applications in physics and engineering. 

A related theory for combinatorial covariant exterior derivatives (but not products) has been recently developed by Christiansen and Hu~\cite{ChHu2023}. In~\S\ref{sec:coarse} we show that their constructions can be extracted by coarsening objects and operators of our theory to unordered simplicial complexes. 
Some aspects of~\cite{ChHu2023} can be viewed as a new approach to simplicial gauge theory and this is the latest in a series of papers on computational gauge theory and related topics by Christiansen and co-authors which combine ideas from lattice gauge theory and finite element methods~\cite{Christiansen2009, ChHa2009, ChHa2011a, ChHa2012}. A recent finite element based approach to discretization of Levi-Civita connection is in~\cite{BeGa2022}.

Dimakis and M\"uller-Hoissen~\cite{DiMu1994} develop differential algebras on discrete sets as a foundation for discrete differential calculus and gauge theory. A calculus for simplicial complexes can be derived as a special case. They do not consider questions of naturality under simplicial maps. Skopnekov~\cite{Skopenkov2023}  discretizes classical field theories from a Lagrangian relying on cochains, coboundaries, and cap products and establishes a discrete Noether theorem. 

Braune and co-authors~\cite{BrToGaDe2024} build on an earlier version of the present paper. They introduce a detailed framework for discretizing bundle-valued forms using a local trivialization. Their construction is an effective realization of Construction~\ref{construction2} using parallel propagated frames and a choice that makes the error in discrete Stokes' theorem vanish in the limit. They also introduce an alternation operator to compose with our $d_\nabla$ and related products showing various numerical and theoretical convergence results.

\section{Background}
\label{sec:background}
\subsection{Locally ordered simplicial complexes}\label{subsec:simplicial}
The material in this section is standard e.g.,~\cite{Munkres1984,FrPi1990}. An \emph{abstract simplicial complex} $X$ is a set of finite sets such that if $\tau\in X$ and $\sigma \subset \tau$ then $\sigma \in X$. Each element of $X$ is called a \emph{simplex}. If the vertices of $X$ are points in $\R^N$ for some $N\ge 0$ and vertices in each simplex are affinely independent then $X$ is also a \emph{piecewise-linear (PL) Euclidean simplicial complex}\footnote{Every PL Euclidean simplicial complex is also an abstract one and we will mean either type if there is no confusion. Most applications we envision take place on PL Euclidean simplicial complexes but this is not a requirement and all definitions of discrete operators are valid for abstract ones unless specified.}. A manifold simplicial complex is a PL Euclidean simplicial complex that is a topological manifold. 

A simplex $\sigma$ with $k+1$ elements has \emph{dimension} $\dim \sigma = k$ and called a $k$-\emph{simplex}. 
The set of all $k$-simplices of $X$ is $X_k \defeq \setdef{\sigma \in X}{\dim \sigma = k}$ and the $k$-\emph{skeleton} of $X$ is $X^{(k)} \defeq \setdef{\sigma\in X}{\dim \sigma \le k}$. Each $X^{(k)}$ is a \emph{subcomplex} of $X$, a subset which is also a simplicial complex. A simplex $\tau \subset \sigma$ is a \emph{face of} $\sigma$ denoted $\tau \prec \sigma$ or $\sigma \succ \tau$, and to allow for $\tau=\sigma$, $\tau \preceq \sigma$ or $\sigma \succeq \tau$. A \emph{maximal} simplex is not a face of any other simplex in $X$. We refer to abstract or Euclidean 0-, 1-, 2-, and 3-simplices as vertices, edges, triangles and tetrahedra, respectively. As a $k$-simplex is determined by its vertices, we often use notation as $\sigma = \{v_0,\ldots,v_k\}$. We are only interested in \emph{finite complexes}, meaning $X$ is a finite set. In particular, the dimension $\max_{\sigma \in X}\dim \sigma$ is well-defined and finite. 

An \emph{abstract simplicial map} $f:X\to Y$ between simplicial complexes is a function determined by a vertex map $f_0:X_0\to Y_0$ with the proeprty that $\sigma \in X$ implies $f(\sigma) \in Y$. Abstract simplicial complexes and simplicial maps form the objects and morphisms of a category denoted $\SiCo$~\cite{FrPi1990}. 

 An \emph{orientation} of a simplex $\sigma \in X$ is a choice of equivalence class of vertex ordering, where two choices of ordering are in the same equivalence class if they are related by an even permutation. For a fixed orientation of $\sigma$, we use the notation $-\sigma$ to denote the other orientation. For example, for a permutation $\pi\in S_{k+1}$ we have $[v_0,\ldots,v_k] = \sgn(\pi) [v_{\pi(0)},\ldots,v_{\pi(k)}]$. When orientations of $\tau^k,\sigma^k$ can be compared then the \emph{relative orientation} is $o(\tau^k,\sigma^k) = +1$ or $-1$ according to if orientations agree or not. In particular for $\sigma^{k+1} = \simplex[v_0]{v_{k+1}}$, the $i$th face $\sigma^k = \face[\hat{v}_i]{v_0}{v_{k+1}}$ has the \emph{induced orientation} $o(\sigma^k,\sigma^{k+1}) = (-1)^i$. 

To set up a locally ordered simplicial complex, let $V$ be a finite set and $R \subseteq V\times V$ a binary reflexive antisymmetric relation on $V$. Consider a set $X$ of subsets $\sigma \subseteq V$ such that $R\;\cap \;\sigma \times \sigma$ is a total order on $\sigma$. Then $X$ is a simplicial complex and the pair $(X,R)$ is called a \emph{(locally) ordered simplicial complex}~\cite[page 111]{FrPi1990}. 
We observe that $R$ determines an ordering of the vertices of each simplex in $X$; but vertices of distinct simplices, even those of different faces of a given simplex, need not be comparable. Typically we omit $R$ from the notation: hereafter $X$ will denote a locally ordered simplicial complex, and for $\sigma, \tau \in X$ with $(\sigma,\tau)\in R$ we write $\sigma \leq \tau$. 

For locally ordered simplicial complexes $X$ and $Y$, an \emph{order-preserving simplicial map} $f\colon X\to Y$ is a simplicial map with the property that $\sigma\leq \tau$ implies $f(\sigma) \leq f(\tau)$. Locally ordered simplicial complexes and order-preserving maps form the objects and morphisms of a category denoted $\OSiCo$ with a functor $\OSiCo\to \SiCo$ that forgets the ordering~\cite{FrPi1990}. There is no obvious functor in the opposite direction: for $\X$ and $\Y$ (unordered) simplicial complexes and $f\colon \X\to \Y$ a simplicial map, choices of ordering on $\X$ and $\Y$ are generally not compatible with $f$. However, as described below there is a subdivision procedure that \emph{is} functorial.

\subsection{Subdivision functor} \label{subsec:subdiv}

For $\X$ an (unordered) simplicial complex of dimension $n$, its \emph{subdivision} $X:=\sd \X$ has as vertex set the simplices of $\X$, and whose maximal simplices are the set
$\setdef{\{\sigma^i, \sigma^{i-1}, \ldots, \sigma^0\}}{\sigma^i \succ \sigma^{i-1} \succ \ldots \succ \sigma^0}$, for all maximal $\sigma^i \in \X$, $0\le i \le n$. The superscripts denote dimension. The vertices of each simplex in $X$ have a canonical total order given by dimension and thus $X$ is an ordered simplicial complex. Furthermore, for any simplicial map $f\colon \X\to \Y$ the map $\sd f: \sd \X \to \sd \Y$ is an order-preserving  simplicial map~\cite{FrPi1990} where the values of $\sd f$ are the obvious ones determined by $f$. 

\begin{remark}\label{rem:subdiv-notation}
    Simplices in $X := \sd \X$ are sets of sets and vertices in $X$ are those sets which are simplices of $\X$. For readability, a vertex of $X$ corresponding to $\sigma \in \X$ will be labeled $c_\sigma$ or $c(\sigma)$. In a geometric simplicial complex $c(\sigma)$ can be identified with a point associated with $\sigma$, such as a barycenter or circumcenter if such a center can be defined, or some arbitrary unspecified point, typically (but not necessarily) in the interior of $\sigma$. For example, for $\X$ simplicial complex of triangle $\{0,1,2\}$, one of the triangles in $X$ is $\{\{0,1,2\},\{1,2\},\{2\}\}$. An oriented version is $[\{0,1,2\},\{1,2\},\{2\}]$. A vertex of $X$ such as $\{0,1\}$ will be denoted $c(\{0,1\})$ or $c_{\{0,1\}}$ or $c_{01}$ or even just $01$. Thus  $\{\{0,1,2\},\{1,2\},\{2\}\} \in X$ might be denoted $\{c_{012}, c_{12}, c_2\}$ and the oriented version $[c_{012}, c_{12}, c_2]$, or dispensing with the commas $[c_{012}\, c_{12}\, c_2]$, or simply $[012, 12, 2]$. 
\end{remark}

\begin{remark}\label{rem:subdiv-vertex-order}
   The choice of ordering vertices of $\sd \X$ by decreasing dimension has an important consequence for later definitions. In particular it is a useful fact that \emph{all} $k$-simplices in $\sd \X$ that result from the subdivision of $\sigma^k \in \X$ have $c(\sigma^k)$ as the smallest vertex. A typical such simplex is $\{c(\sigma^k),c(\sigma^{k-1}),\ldots,c(\sigma^0)\}$ where $\sigma^k \succ \sigma^{k-1} \succ \ldots \succ \sigma^0$.
\end{remark}

\begin{remark}\label{rem:c-notation}
    For $0 \le i \le j \le n$ and $\sigma^i \preceq \sigma^j$ we will sometimes use the shorthand notation $c_{\sigma^i \preceq \sigma^j} = c[\sigma^i \preceq \sigma^j] := [c(\sigma^i), c(\sigma^{i+1}), \ldots, c(\sigma^j)]$ with $\sigma^i \preceq \sigma^{i+1} \preceq \ldots \preceq\sigma^j$ and $c_{\sigma^i \succeq \sigma^j} = c[\sigma^i \succeq \sigma^j] := [c(\sigma^j), c(\sigma^{j-1}), \ldots, c(\sigma^i)]$ with $\sigma^j \succeq \sigma^{j-1} \succeq \ldots \succeq \sigma^i$ for oriented simplices in $\sd \X$.       
\end{remark}

\subsection{Duals and cubes}\label{sbsc:duals-cubes}
For each pair $\sigma^k \preceq \sigma^{k+l}$ in $\X$ for $0\le k,l, k+l\le n$ there are $l!$ dimension $l$ simplices in $\sd \X$ of the form $\{c(\sigma^k), c(\sigma^{k+1}),\ldots, c(\sigma^{k+l})\}$ with $\sigma^k \preceq \sigma^{k+1} \preceq \ldots \preceq \sigma^{k+l}$. We will call the cell formed by the union of these, and oriented as described below, the \emph{dual of} $\sigma^k$ \emph{in} $\sigma^{k+l}$ denoted $D(\sigma^k\preceq\sigma^{k+l})$. In DEC, for $\X$ a manifold simplicial complex of dimension $n$ the cell $D(\sigma^k \preceq\sigma^n)$ is called the \emph{dual cell} of the \emph{primal} simplex $\sigma^k$ and denoted $\star\sigma^k$. The cells $D(\sigma^k\preceq\sigma^{k+l})$ are combinatorial cubes of dimension $l$ with $2^l$ vertices with vertex set $\cup_{i=k}^{k+l} \setdef{c(\sigma^i)}{\sigma^k \preceq \sigma^i \preceq \sigma^{k+l}}$. 
The \emph{cubical cell complex} associated with $\X$, denoted $\Q(\X)$, is the CW complex formed by the combinatorial cubes $D(\sigma^k\preceq\sigma^{k+l})$, $0 \le k,l,k+l \le n$. It is a refinement of $\X$ like the subdivision $\sd \X$ except that the cells are combinatorial cubes constructed from simplices of $\sd \X$. 

Since $D(\sigma^k\preceq\sigma^{k+l})$ are dual cells they can be oriented by the algorithm in ~\cite{Hirani2003}, which we recall here. To orient such an $l$-dimensional cube each $l$-simplices of $\sd \X$ in this cube should be oriented as
\begin{equation}\label{eq:dual-orientation}
o(c[\sigma^0\preceq \sigma^k],\, \sigma^k)\;\;
o(c[\sigma^0\preceq \sigma^{k+l}],\, \sigma^{k+l})
\;\;\simplex[c(\sigma^k)]{c(\sigma^{k+l})}\, ,
\end{equation}
using the shorthand notation of Remark~\ref{rem:c-notation}. If the simplices $c[\sigma^i\preceq \sigma^j]$ in~\eqref{eq:dual-orientation} are written in the reverse order $c[\sigma^j \succeq \sigma^i]$ then an overall sign factor of $(-1)^{kl}$ is needed. It can be shown that all $l$-simplices of $D(\sigma^k\preceq\sigma^{k+l})$ are oriented consistently by~\eqref{eq:dual-orientation}. Furthermore all such $(n-k)$-simplices in $D(\sigma^k\preceq\sigma^n)$ (that is, in $\star \sigma^k$) are oriented consistently.

For each $k$-simplex $\sigma^k \in \X$, the union of simplices of the form $\{c(\sigma^0), \ldots,c(\sigma^k),\ldots,c(\sigma^n)\}$ with $\sigma^0\preceq \ldots \preceq \sigma^k \preceq \ldots \preceq \sigma^n$ is the \emph{support volume} of $\sigma^k$ denoted $V_{\sigma^k}$. The orientation rule~\eqref{eq:dual-orientation} gives consistent orientation to all the $n$-simplices which matches the orientation of $\sigma^n$. In a Euclidean simplicial complex the support volumes are the intrinsic convex hull of a simplex and its dual cell and tile the same underlying space as the original complex.

\begin{remark}\label{rem:cube-vertex-order}
    Similar to Remark~\ref{rem:subdiv-vertex-order} the choice of ordering vertices of $\sd \X$ in decreasing dimension order also has implications for the cubical refinement. For any $\sigma^k \preceq \sigma^{k+l}$ in $\X$, all of the $l!$ simplices in the cube $D(\sigma^k\preceq\sigma^{k+l})$ have the same smallest vertex $c(\sigma^{k+l})$ and the same largest vertex $c(\sigma^k)$.
\end{remark}

\subsection{Discrete exterior calculus and naturality}\label{sec:DEC}
The input data for DEC is a simplicial complex $X$ with additional decorations and properties. The discrete notions of differential form, exterior derivative, and wedge product only depend on the simplicial complex, importing standard methods from simplicial algebraic topology. When incorporating features that depend on a metric (e.g., a discretization of the Hodge star operator) essentially one requires that $X$ approximates a manifold. This assumption is appropriate given the desired applications: DEC has been used mostly as a method for solving partial differential equations (PDEs) on simplicial approximations of embedded orientable manifolds. Riemannian metric is encoded in DEC via a primal and dual cell complex incorporating orthogonality and lengths, areas, volumes etc. 

Given a differential form $\alpha \in \Omega^k(M)$, its discrete analog in DEC is the $k$-cochain $\int_{\rule{0.5em}{0.4pt}} \alpha$ which takes values in $\mathbb{R}$ when evaluated on $k$-dimensional chains, that is, the result of a de~Rham map~\cite{Dodziuk1976}. The space of real-valued $k$-cochains on simplicial complex $X$ are denoted $C^k(X)$. The coboundary operator on cochains plays the role of discrete exterior derivative ($d$) and the antisymmetrized cup product plays the role of a discrete wedge product ($\wedge$) and $d$ satisfies a graded Leibniz for this product. The wedge is also graded commutative but associative only up to homotopy and this leads to interesting algebraic structures in DEC~\cite{Liu2026}. A discrete contraction, Lie derivative, and a discrete Hodge star have also been defined but are not relevant in this paper. See~\cite{Hirani2003} for details. The discrete Hodge star construction involves a Poincar\'e dual complex of $X$ using circumcenters, and this will likely prove useful in future work in light of the coarsening described in~\S\ref{sec:coarse}. The discrete $d$ and $\wedge$ commute with pullback by abstract simplicial maps. Thus such maps play the role that smooth maps play in calculus on smooth manifolds.

\section{Discrete Vector Bundles with Connection}\label{sec:dscrtvctrbndls}

\subsection{Discrete vector bundles with connection and vector bundle-valued cochains}
\begin{definition}\label{def:dvbwc}
  For a locally ordered simplicial complex $X$, $E\to X$ a \emph{real} (respectively, \emph{complex}) \emph{discrete vector bundle with connection over $X$} is data:
  \begin{enumerate}
  \item for each vertex $i\in X_0$, a finite-dimensional real (respectively, complex) vector space $E_{i}$ called the \emph{fiber} at~$i$; and
  \item for each edge $[i,j] \in X_1$, an invertible linear map $U_{ji}\colon E_i\to E_j$ called \emph{parallel transport} from $i$ to $j$. 
  \end{enumerate}
These data are required to satisfy the compatibility condition $U_{ij}=U_{ji}^{-1}$. 
\end{definition}

\begin{remark}
The above definition only depends on the underlying unordered simplicial complex, but most of the desired constructions that follow (e.g., bundle-valued $k$-cochains and the exterior covariant derivative) make explicit use of the local ordering. We also note that a discrete vector bundle without connection is not particularly useful notion: this would yield vector spaces~$E_i$ at each vertex with no geometric relationship with each other. As connections will always be present, we sometimes refer to a discrete vector bundle with connection simply as a \emph{discrete vector bundle}.
\end{remark}

The \emph{rank} of a vector bundle on a connected component of $X$  is the (necessarily constant) dimension of the fibers. Given a subcomplex $Y\subset X$, the \emph{restriction} of a bundle is a discrete vector bundle with connection gotten from the obvious restriction of the data (1) and (2) above. 

One is always free to choose a basis for the vector space at each fiber, giving isomorphisms $\R^n\cong E_i$ or $\complex^n\cong E_i$ for each~$i$. Borrowing terminology from the physics literature, we refer to choices of such isomorphisms as a choice of \emph{gauge}. After a choice of gauge has been made, the parallel transport maps are determined by matrices in $\GL_n(\R)$ for real vector bundles or $\GL_n(\complex)$ for complex vector bundles. Below, we use the notation $\GL_n$ to denote either $\GL_n(\R)$ or $\GL_n(\complex)$, i.e., for statements that hold over both $\R$ and $\complex$.

For a simplex $\sigma$ in a locally ordered simplicial complex $X$, the \emph{origin vertex} of $\sigma$ is its lowest vertex according to the ordering. Thus each simplex has a fiber associated with it, the one at the origin vertex. The fiber at a vertex $v$ is shared by all simplices for which $v$ is the lowest vertex. 

\begin{definition}\label{def:cochain}
  A \emph{vector bundle valued $k$-cochain} $\alpha$ assigns to each oriented $k$-simplex $\sigma \in X$ with origin vertex $l$ a vector $\eval{\alpha}{\sigma}_l\in E_l$. For $\pi\in S_{k+1}$ a permutation, $\sigma = [v_0, \ldots, v_k]$, $\eval{\alpha}{[v_{\pi(0)}, \ldots, v_{\pi(k)}]}_l := \sgn(\pi) \eval{\alpha}{\sigma}_l$ where $\sgn(\pi)$ is the sign of $\pi$. The vector space of $k$-cochains is denoted $C^k(X;E)$. A \emph{section} $s$ is a vector bundle valued 0-cochain, that is, a map that assigns a vector $s_i\in E_i$ for each vertex.  
  \end{definition}

\begin{definition} \label{def:connection}
  The discrete \emph{covariant derivative}  $\nabla\colon C^0(X;E)\to C^1(X;E)$ is characterized by 
\begin{equation}
 \eval{\nabla s}{[i,j]}_i  := U_{ij} s_j-s_i\in E_i,\qquad i<j, \quad s\in C^0(X;E)=\Gamma(E). \label{eq:connection} 
\end{equation}
\end{definition}

\begin{remark}
We observe that the parallel transport maps (2) in Definition~\ref{def:dvbwc} are equivalent data to the covariant derivative $\nabla$. As such, we often use the notation $(E,\nabla)$ for a discrete vector bundle with connection. 
\end{remark}

\begin{definition} \label{def:bundlemaps}
 Let $X$ and $Y$ be locally ordered simplicial complexes and $f\colon Y\to X$ an order preserving abstract simplicial map. Given $(E,\nabla) \to X$ and $(E',\nabla') \to Y$ a \emph{map of discrete vector bundles covering $f$}, denoted $\bar{f}\colon (E',\nabla')\to (E,\nabla)$, is a collection of linear maps $\bar{f}_l\colon E_l'\to E_{f(l)}$, one for each $l \in Y_0$ so that the following diagram commutes
\begin{center}
  \begin{tikzcd}
  E_i' \arrow[r, "\bar{f}_i"] \arrow[d, "U_{ij}"] &  E_{f(i)} \arrow[d, "U_{f(i)f(j)}"]\\
  E_j' \arrow[r, "\bar{f}_j"] & E_{f(j)}
\end{tikzcd}
\end{center}
whenever there is an edge from vertex $i$ to $j$. In particular, if an edge $[i,j]$ in $Y$ is sent to a vertex in~$X$ under $f$ (so $f(i) = f(j)$), the parallel transport $U_{f(i) f(j)}=\Id_{E_{f(i)}}$ is required to be the identity.
\end{definition}

\begin{definition} \label{def:pullback-bundle}
  Given $(E,\nabla)\to X$ and $f\colon Y\to X$ an order-preserving simplicial map, the \emph{pullback bundle} over $Y$ is denoted $f^\ast(E,\nabla)=(f^*E,f^*\nabla)$ and defined as follows. Take as fibers $(f^*E)_i=E_{f(i)}$ for each $i\in Y_0$, and the connection $f^*\nabla$ is defined by assigning to each edge $[i,j] \in Y_1$ the parallel transport map:
  \[
    U_{ij} = 
    \begin{cases}
      U_{f(i)f(j)} & \text{if $[f(i), f(j)] \in X_1$}\\
      \Id_{E_{f(i)}} = \Id_{E_{f(j)}}  & \text{otherwise.}
    \end{cases}
  \]
  There is an evident map of discrete vector bundles with connection $\bar{f}\colon (f^*E,f^*\nabla) \to (E,\nabla)$ covering $f\colon Y\to X$ defined fiberwise by $\bar{f}_i = \bar{f}\vert_{E_i}\colon (f^*E)_i \to E_{f(i)}$ as the identity map. A pullback covering an isomorphism of simplicial complexes is called an \emph{isomorphism of discrete bundles with connection}, and an isomorphism cover the identity map on $X$ is called an \emph{automorphism} of $(E,\nabla) \to X$. If a choice of gauge has been fixed, then an automorphism is also called a \emph{gauge transformation}, specified by an element of $\GL_n$ at each vertex. 
\end{definition}

The pullback in discrete vector bundles satisfies the analogous universal property of the pullback of vector bundles over smooth manifolds; the proof is straightforward and left to the reader. 

\begin{proposition}
  Given $(E^\prime,\nabla^\prime) \to Y$ and $(E,\nabla) \to X$ and  $(E^\prime,\nabla)\to (E,\nabla)$ a map of discrete vector bundles covering simplicial map $f: Y\to X$ there is a unique map $(E',\nabla^\prime)\to (f^*E,f^*\nabla')$ over identity map of $Y$.
\end{proposition}

The above universal property uniquely characterizes the pullback $f^*(E,\nabla)$ up to unique isomorphism. This allows one to verify many of the standard properties of pullbacks using the same arguments as for vector bundles on smooth manifolds. For example, for maps $f\colon Y\to X$, $g\colon Z\to Y$ and $(E,\nabla)$ over $X$ there is a unique isomorphism between $g^*f^*(E,\nabla)$ and~$(f\circ g)^*(E,\nabla)$ over $Z$. 

We define pullbacks of $E$-valued cochains as the cochain dual of the corresponding chain map.
\begin{definition} \label{def:pullback-cochain}
  Let $Y$ and $X$ be locally ordered simplicial complexes, $f\colon  Y\to X$ an order-preserving abstract simplicial map, and $\alpha\in C^k(X; E)$. Then the \emph{pullback of $\alpha$}, denoted $f^*\alpha$, is the $f^*E$-valued cochain defined by:
  \[
    \eval{f^*\alpha}{\simplex{k}}_0 \defeq
    \begin{cases}
      \big\langle\alpha,\, f(\simplex{k})\big\rangle_{f(0)} &
      \text{if $f(\simplex{k})$ is a $k$-simplex in $Y$}.\\
      0 & \text{otherwise}\, .
    \end{cases}
  \]
\end{definition}

We remind that $f(0)$ is the origin vertex of $f(\simplex{k})$ since $f$ is order-preserving. 

\subsection{Discretization of smooth connections}\label{sec:form-discretization}
To provide a link with smooth geometry we provide a choice of discretization scheme that can be applied to a smooth vector bundle with connection $(E,\nabla)\to M$, and $X$ an ordered simplicial complex gotten from a (curvilinear) triangulation of the smooth manifold~$M$. This scheme is by no means unique, as we comment below. 

\begin{Construction}[Discretizing a smooth connection]\label{construction1}
With input data $(E,\nabla)\to M$ and $X$ as above, define a discrete vector bundle with connection over~$X$ as follows. For each vertex $i\in X_0$ assign the fiber $E_i$ of $E$, identifying the vertex with the corresponding point of $M$. For each edge $[i,j]\in X_1$, define $U_{ji}\colon E_i\to E_j$ as parallel transport along the associated path in $M$.

The above assignment is (contravariantly) functorial: for a smooth map $f\colon N\to M$ that induces a map $Y\to X$ of triangulations, there is a canonically induced map of discrete vector bundles with connection: $f^*$ determines a linear isomorphisms on the fibers denoted $\bar{f}_i\colon (f^*E)_i\to E_{f(i)}$, and the naturality of parallel transport in smooth geometry guarantees the equality~$U_{f(j)f(i)}=\bar{f}_j\circ U_{ji}\circ \bar{f}^{-1}_i$.
\end{Construction}

\begin{remark}
    In the above approach, the discrete covariant derivative~\eqref{eq:connection} is the forward difference approximation to the smooth covariant derivative operator.
\end{remark}

Now we turn to discretization of $E$-valued $k$-forms. In contrast to the de~Rham map for ordinary smooth differential forms, for a (nontrivial) bundle $E$ there is no canonically defined integral of $\hat{\alpha}\in \Omega^k(M;E)$ over $k$-dimensional domains in~$M$. Instead, one can integrate $\hat{\alpha}$ in any local trivialization of $E$; the resulting value will depend on the choice of trivialization.

\begin{Construction}[Discretizing bundle-valued $k$-forms]\label{construction2}
For a $k$-simplex $\mathbb{R}^k \supseteq \sigma \hookrightarrow M$ with origin vertex $l$, the linear contracting homotopy from $\sigma$ to $l$ gives a preferred trivialization $E\vert_\sigma \simeq \sigma \times E_l$ over~$\sigma$ with trivializing fiber $E_l$. This permits the definition
\begin{equation}\label{eq:discretize}
\eval{\alpha}{\sigma}_l :=\int_\sigma \hat{\alpha}\in E_l,
\end{equation}
of the discrete $k$-cochain $\alpha \in C^k(X; E)$ as the discretization of $\hat{\alpha}$. 
A sophisticated and careful generalization of this idea is carried out in~\cite{BrToGaDe2024}.
\end{Construction}

\begin{remark}
There are two choices in the above construction: (i) an origin vertices~$l\in \sigma$ for each $\sigma\hookrightarrow M$ and (ii) the choice of linear contracting homotopy from $\sigma$ to $l$ in each simplex. The first choice is natural with respect to smooth maps $N\to M$ inducing ordered simplicial maps $Y\to X$. The second choice makes naturality more difficult in this context: an arbitrary smooth map need not be compatible with the linear contracting homotopy. Note however that the construction is natural for piecewise linear maps of ordered simplicial complexes. 
\end{remark}

\begin{remark}\label{rmk:curvature2}
We comment on our two main motivations for considering discretization in the above way. First, when the smooth bundle is flat, we recover the \v{C}ech complex with local coefficients in~$E$, and in turn a complex computing the $E$-twisted de~Rham cohomology of $M$, see~\S\ref{sec:flat}. Second, the discretization of the curvature 2-form in the above procedure leads to the integrals~\eqref{eq:Taylor}, which has a classical interpretation in term of holonomy. An a priori different construction of discrete curvature as a square of the discrete covariant derivative (see Definition~\ref{def:discrete-curvature}) turns out to give essentially the same result, modulo the stated error term; see Remark~\ref{rmk:holvscuvature}. In particular, the error goes to zero as the triangulation is refined. 
\end{remark}

We return to flat bundles and cohomology with local coefficients in~\S\ref{sec:flat}.

\begin{remark}
 Other discretization schemes exist, though all approaches depend on various choices, e.g.,~\cite{BrToGaDe2024} build on earlier versions of our paper and develop a scheme with certain convergence properties made possible by an inspired choice of trivialization and an alternation operation composed with our $d_\nabla$. At present, such alternation seems to obscure some of the discrete geometric structures. In the end, understanding which choice of discretization is ``best" will ultimately require a deeper understanding of the desired applications. With this in mind, for the remainder of the present paper we will focus on the combinatorial and algebraic properties of discrete vector bundle-valued cochains, leaving the questions of discretization schemes and their convergence properties to future work. 
\end{remark}

\begin{remark}
There are discrete vector bundles simplicial complexes that do not arise as the discretization of any smooth vector bundle with connection. For example, consider the real line bundle on the 2-simplex with parallel transports on edges given by $-1$, $+1$, and $+1$. If this line bundle was the discretization of a smooth bundle (where the outer loop bounds a disk), the curvature would be required to vanish. Equivalently, the product of parallel transports around the 2-simplex would need to be $+1$. On the other hand, the discrete vector bundle returns~$-1$ for this holonomy. There are important applications where such discrete bundles that do not arise as discretizations of smooth objects play a central role, e.g., in toric code~\cite{Kitaev2006}. These applications give another reason to focus on the discrete theory below. 
\end{remark}

\subsection{Cup Product and Naturality}\label{subsec:cup}

\begin{definition}\label{def:cup}
  Given a vector bundle valued cochain $\alpha \in C^k(X; E)$ and scalar-valued cochain $w \in C^l(X)$ their \emph{cup product} $\alpha \smile w\in C^{k+l}(X;E)$ is defined by its evaluation on a $(k+l)$-simplex at the origin vertex as
  \begin{equation} \label{eq:cup}
    \eval{\alpha \smile w}{\simplex{k+l}}_0 :=
    \eval{\alpha}{\simplex{k}}_0\, \eval{w}{\simplex[k]{k+l}}
  \end{equation}
\end{definition}

From the definition of the cup product and pullbacks one has the following naturality result.

\begin{proposition}[Naturality of cup product]\label{prop:cup-naturality}
    Given $X,Y$ locally ordered simplicial complexes, bundle $(E,\nabla)\to Y$, and abstract simplicial map $f\colon X\to Y$, for any $\alpha \in C^k(Y; E)$ and $w \in C^l(Y)$ and simplex $\simplex[0]{k+l}$ in $X$
    \begin{equation}\label{eq:cup-naturality}
        \eval{f^\ast(\alpha \smile w)}{\simplex[0]{k+l}}_{0}=
        \eval{f^\ast\alpha \smile f^\ast w}{\simplex[0]{k+l}}_{0}
    \end{equation} 
\end{proposition}

\begin{proof}
    Naturality follows since
    \begin{align*}
        \eval{f^\ast(\alpha \smile w)}{\simplex[0]{k+l}}_{0} &=
        \eval{\alpha \smile w}{\simplex[f(0)]{f(k+l)}}_{f(0)}\\
        &= \eval{\alpha}{\simplex[f(0)]{f(k)}}_{f(0)}\;\;
        \eval{w}{\simplex[f(k)]{f(k+l)}}\\
        &= \eval{f^\ast\alpha}{\simplex[0]{k}}_{0}\eval{f^\ast w}{\simplex[k]{k+l}}\\
        &= \eval{f^\ast\alpha \smile f^\ast w}{\simplex[0]{k+l}}_{0}\, .
    \end{align*}
    For dimensional reasons, both sides are 0 if the vertex map of $f$ is not a bijection.
\end{proof}

\begin{remark}
For scalar-valued discrete cochains there are two choices of product: the cup product as in~\eqref{eq:cup} or an (anti-symmetrized) wedge product~\cite{Hirani2003}. These cochain-level operations yield the same product on cohomology. In the bundle-valued setting, the anti-symmetrized wedge product is not well-defined, with an obstruction coming from the curvature, see~\S\ref{sec:crvtr_obstrctn}. In this sense, the cup product~\eqref{eq:cup} is the only viable one between scalar-valued and vector bundle-valued cochains. 
\end{remark}

\section{The discrete exterior covariant derivative}
\label{sec:dnabla}

Following the paradigm from the smooth setting~\eqref{eq:sequence}, we extend the discrete covariant derivative~\eqref{eq:connection} to a discrete exterior derivative on vector bundle-valued cochains. 

\begin{definition}
 Let $\alpha\in C^{k-1}(X; E)$ be a $(k-1)$-cochain. The discrete \emph{exterior covariant derivative} of $\alpha$ is the $k$-cochain $d_\nabla\alpha\in C^k(X;E)$ characterized by
  \begin{equation} \label{eq:dnabla}
    \eval{d_\nabla \alpha}{\simplex{k}}_0 \defeq U_{01} \eval{\alpha}{\simplex[1]{k}}_1+
    \sum_{i=1}^k (-1)^i \eval{\alpha}{\face{0}k}_0.
  \end{equation}
\end{definition}

\begin{lemma}
[Leibniz rule for sections]\label{prop:Leibniz1}
For any $f \in C^0(X)$, section $s \in C^0(X; E)$ and edge $[i,j]$ in $X$ with origin vertex $i$, the discrete covariant derivative satisfies the following Leibniz rule
  \begin{equation} \label{eq:nabla-Leibniz}
    \eval{\nabla (f\smile s)}{[i,j]}_i =
    \eval{d f\smile s+f \smile \nabla s}{[i,j]}_i\, .
  \end{equation}
\end{lemma}
\begin{proof}
  The left side of~\eqref{eq:nabla-Leibniz} is $U_{ij} ( f_j s_j) - f_i s_i$, whereas the right side is the sum of the terms
  \begin{align*}
    \eval{df \smile s}{[i,j]}_i &= \eval{df}{[i,j]} U_{ij} s_j = (f_j - f_i) U_{ij}s_j\\
    \eval{f\smile \nabla s}{[i,j]}_i &= f_i\eval{\nabla s}{[i,j]}_i = f_i(U_{ij} s_j - s_i). 
  \end{align*}
  The result follows. 
\end{proof}

\begin{proposition}[Leibniz rule]\label{prop:Leibniz}
  For $\alpha\in C^{k}(X; E)$ and $w \in C^l(X)$, the discrete covariant derivative $d_\nabla$ satisfies the following Leibniz rule 
  \begin{equation}\label{eq:Leibniz-cup}
    \eval{d_\nabla (\alpha\smile w)}{\simplex{k{+}l{+}1}}_0   =
    \eval{d_\nabla \alpha \smile w}{\simplex{k{+}l{+}1}}_0 +
    (-1)^k \eval{\alpha\smile dw}{\simplex{k{+}l{+}1}}_0.
  \end{equation} 
\end{proposition}
\begin{proof} By definition of $d_\nabla$, we have
  \begin{equation*}\label{eq:Leibniz-cup-step1}
    \eval{d_\nabla (\alpha\smile w)}{\simplex{k{+}l{+}1}}_0   = U_{01} \eval{\alpha\smile w}{\simplex[1]{k{+}l{+}1}}_1+
    \sum_{i=1}^{k{+}l{+}1} (-1)^i \eval{\alpha\smile w}{\face{0}{k{+}l{+}1}}_0\, 
  \end{equation*}
and evaluating the cup products above yields
  \begin{multline}\label{eq:Leibniz-cup-lhs}
    U_{01} \eval{\alpha}{\simplex[1]{k+1}}_1\; \eval{w}{\simplex[k{+}1]{k{+}l{+}1}}  + 
    \sum_{i=1}^{k} (-1)^i \eval{\alpha}{\face{0}{k{+}1}}_0\;
    \eval{ w}{\simplex[k{+}1]{k{+}l{+}1}} +\\
    \sum_{i=k{+}1}^{k{+}l{+}1} (-1)^i \eval{\alpha}{\simplex{k}}_0\; \eval{w}{\face{k}{k{+}l{+}1}}\, .
  \end{multline}
The first term on the right side of~\eqref{eq:Leibniz-cup} expands to
  \[
    U_{01}\eval{\alpha}{\simplex[1]{k{+}1}}_1 \; \eval{w}{\simplex[k+1]{k{+}l{+}1}}
    +\sum_{i=1}^{k+1} (-1)^i\eval{\alpha}{\face{0}{k{+}1}}_0\; \eval{w}{\simplex[k{+}1]{k{+}l{+}1}}\, .
  \]
  In preparation for a cancellation we will separate out the last term in the summation above, which yields
  \begin{multline}\label{eq:Leibniz-cup-rhs1}
    U_{01}\eval{\alpha}{\simplex[1]{k{+}1}}_1 \; \eval{w}{\simplex[k{+}1]{k{+}l{+}1}}\; +\\
    \sum_{i=1}^k(-1)^i\eval{\alpha}{\face{0}{k{+}1}}_0\; \eval{ w}{\simplex[k{+}1]{k{+}l{+}1}}
     +(-1)^{k+1}\eval{\alpha}{\simplex{k}}_0\; \eval{ w}{\simplex[k{+}1]{k{+}l{+}1}}\, .
  \end{multline}
The second term of the right hand side of~\eqref{eq:Leibniz-cup} expands to
  \[
    (-1)^k\eval{\alpha}{\simplex{k}}_0\; \eval{dw}{\simplex[k]{k{+}l{+}1}} =
    (-1)^k\eval{\alpha}{\simplex{k}}_0\; \sum_{i=k}^{k+l+1}(-1)^{(i-k)} \eval{w}{\face{k}{k{+}l{+}1}}\, .
  \]
Separating out the first term in the summation above yields
\begin{multline}\label{eq:Leibniz-cup-rhs2}
  (-1)^k\eval{\alpha}{\simplex{k}}_0\; (-1)^0\eval{w}{\simplex[k{+}1]{k{+}l{+}1}} + \\
  (-1)^k \sum_{i=k+1}^{k+l+1}\eval{\alpha}{\simplex{k}}_0\; (-1)^{(i-k)}
\eval{w}{\face{k}{k{+}l{+}1}}\, .
\end{multline}
On adding~\eqref{eq:Leibniz-cup-rhs1} and~\eqref{eq:Leibniz-cup-rhs2} the last term  in~\eqref{eq:Leibniz-cup-rhs1} and the first term in~\eqref{eq:Leibniz-cup-rhs2} cancel and the result is ~\eqref{eq:Leibniz-cup-lhs}. 
\end{proof}

\begin{proposition}[Naturality of $d_\nabla$] \label{prop:dnaturality}
    Let $X,Y$ be ordered simplicial complexes, $f: X \to Y$ an order preserving abstract simplicial map, and $E\to Y$ a discrete vector bundle with connection. Then for any $\alpha\in C^k(Y; E)$ and $(k+1)$-simplex $\simplex{k+1}$ in $X$:
    \begin{equation} \label{eq:dnaturality}
        \eval{f^\ast d_\nabla\alpha}{\simplex{k+1}}_0 =
        \eval{d_\nabla f^\ast\alpha}{\simplex{k+1}}_0\, .
    \end{equation}
\end{proposition}
\begin{proof}
    By definition, 
    \begin{equation}\label{eq:dnaturality-rhs}
        \eval{d_\nabla f^\ast\alpha}{\simplex{k+1}}_0=U_{f(0),f(1)}\Eval{\alpha}{f\big(\simplex[1]{k+1}\big)}_{f(1)} +
    \sum_{j=1}^{k+1}(-1)^j\; 
    \Eval{\alpha}{f\big([0\ldots\widehat{j}\ldots k+1]\big)}_{f(0)}\, .
    \end{equation}
    Let $f(\simplex{k+1})$ be an $l$-simplex in $Y$. If $l=k+1$,~\eqref{eq:dnaturality} follows by definition since $i\mapsto i$ as $f$ is order preserving. For $l<k$, $f^*\alpha$ evaluates to 0 for all terms in~\eqref{eq:dnaturality-rhs}. The remaining case is $l=k$. In this case due to order preservation two consecutive ordered vertices $i,i+1$ must map to a single vertex $i$ or $i+1$. For both cases, if $i=0$ then the first term in~\eqref{eq:dnaturality-rhs} cancels with the $j=1$ term and other terms are 0 due to a repeated vertex. For $i\ge 1$, whether $i,i+1\mapsto i,i$ or $i,i+1\mapsto i+1,i+1$, all terms in~\eqref{eq:dnaturality-rhs} are 0 except $j=i$ and $j=i+1$. These two terms are equal but with opposite signs and hence cancel.
\end{proof}

\section{Curvature as a homomorphism valued cochain} \label{sec:curvature}

For $(E,\nabla) \to M$ a smooth vector bundle with connection, we recall that the operator $d_\nabla^2 = d_\nabla\circ d_\nabla$ is $C^\infty(M)$-linear, and hence can be identified with an endomorphism-valued 2-form $R$, affording one of the smooth definitions of curvature. We mimic this approach in the discrete setting, which first requires that we make sense of the discrete analog of $d_\nabla^2$. As we shall see, $d_\nabla^2$ is naturally a linear map between \emph{pairs} of fibers, which we formalize as a \emph{homomorphism-valued $2$-cochain}. This ``spreading out'' is a common phenomenon when discretizing smooth objects. 

\subsection{Curvature as a homomorphism-valued 2-cochain}

\begin{definition}\label{def:hom}
  A \emph{homomorphism-valued} $k$-cochain assigns to each $k$-simplex $\simplex{k}$ a linear map $E_k \to E_0$ from the highest to the lowest vertex. The space of homomorphism-valued $k$-cochains is denoted $C^k(X; \Hom(E))$. Given $B \in C^k(X; \Hom(E))$ and $\alpha\in C^l(X; E)$ the action of $B$ on $\alpha$ is $B\smile \alpha$ defined by:
  \begin{equation} \label{eq:homAction}
    \eval{B\; \smile \alpha}{\simplex{k+l}}_0 := \eval{B}{\simplex{k}}_{0,k} \; \eval{\alpha}{\simplex[k]{k+l}}_k\;.
  \end{equation}
\end{definition}
The subscript in $\eval{B}{\simplex{k}}_{0,k}$ indicates that the value is a map $E_k\to E_0$.

\begin{remark}\label{rem:why-no-homE}
    In Definition~\ref{def:hom} we have used $\Hom(E)$ without defining it. We could have defined a vector bundle $\Hom(E)$ over $X$ consisting of the data of a vector space of all linear maps $E_k\to E_0$ for each simplex $\simplex{k}$. Then a homomorphism valued $k$-cochain would have been a section of such a bundle. This is different from the situation of endomorphism-valued bundle $\End(E)$ in the smooth setting because endomorphisms are linear maps $E_p\to E_p$ for each $p\in M$. We will use $\Hom(E)$ just as a notation in the definition of cochains.
\end{remark}

\begin{proposition}\label{prop:dnabla2-in-Hom}
    For $(E,\nabla) \to X$ a discrete vector bundle with connection, the square of the discrete covariant derivative is a homomorphism valued 2-cochain, $d_\nabla^2 = d_\nabla \circ d_\nabla\in C^2(X; \Hom(E))$.
\end{proposition}
\begin{proof}
    Let $s \in C^0(X;E)$. Then $d_\nabla^2 s = d_\nabla \nabla s$ on triangle $[0,1,2]$ is 
    \begin{align*}
        \eval{d_\nabla d_\nabla s}{[0,1,2]}_0 &= 
        U_{01} \eval{d_\nabla s}{[1,2]}_1 - \eval{d_\nabla s}{[0,2]}_0 + \eval{d_\nabla s}{[0,1]}_0 \\
        &= U_{01}(U_{12} s_2 - s_1) - (U_{02} s_2 - s_0) + (U_{01} s_1 - s_0) \\
        &= (U_{01} U_{12} - U_{02}) s_2\, .
    \end{align*}
Thus $d_\nabla^2 \in C^2(X;\Hom(E))$ since its value on $[0,1,2]$ is a linear map $E_2 \to E_0$.
\end{proof}

\begin{definition}\label{def:discrete-curvature}
  Given $(E,\nabla)\to X$ a discrete vector bundle with connection, the \emph{discrete curvature} is defined as $F_\nabla := d_\nabla^2 \in C^2(X;\Hom(E))$. As computed above, the discrete curvature takes values
  \begin{equation}\label{eq:curvature}
    \eval{F_\nabla}{[0,1,2]}_{0,2}= U_{01}U_{12} - U_{02}\, 
  \end{equation}
  which we often denote by $F_{012}:=\eval{F_\nabla}{[0,1,2]}$.
\end{definition}

\begin{remark}\label{rmk:holvscuvature}
One can always arrange a gauge transformation so that $U_{20}$ is the identity matrix. Hence,~\eqref{eq:curvature} recovers the common ``holonomy minus identity'' definition of curvature in (Wilson's) discrete gauge theory. Furthermore, up to this choice of gauge the quantity~\eqref{eq:curvature} agrees with the leading term in the Taylor approximation to the integral of curvature, see Remarks~\ref{rem:taylor} and~\ref{rmk:curvature2}. 
\end{remark}

\begin{remark}
    By a coarsening procedure detailed in~\S\ref{sec:coarse}, discrete curvature will be associated with codimension 2 cells. This agrees with the notion in discrete differential geometry of surfaces~\cite{Kourimska2020} and Regge calculus~\cite{Regge1961}.
\end{remark}

\begin{definition}\label{def:dnabla-hom}
    For $(E,\nabla)\to X$ a discrete vector bundle with connection, the discrete \emph{induced exterior covariant derivative} on homomorphism-valued k-cochains $d_{\nabla^\Hom}\colon C^k(X;\Hom(E)) \to C^{k+1}(X; \Hom(E))$ is characterized by demanding Leibniz rule hold:
    \begin{equation}\label{eq:Hom-Leibniz}
      d_\nabla (B\smile \alpha) = (d_{\nabla^\Hom} B)\smile \alpha + (-1)^k B \smile d_\nabla \alpha\, ,  
    \end{equation}
    for $B\in C^k(X;\Hom(E))$ and $\alpha\in C^l(X;E)$. When applied to degree zero cochains, we use the notation $\nabla^\Hom\colon C^0(X;\Hom(E))\to C^1(X;\Hom(E))$ 
\end{definition}
\begin{proposition}\label{prop:dnabla-hom}
  The homomorphism-valued $(k+1)$-cochain $d_{\nabla^\Hom} B \in C^{k+1}(X;\Hom(E))$ takes values 
  \begin{multline}\label{eq:dnabla-hom}
    \eval{d_{\nabla^\Hom} B}{\simplex{k+1}}_{0,k+1} \defeq U_{01} \eval{B}{\simplex[1]{k+1}}_{1,k+1} +\sum_{i=1}^{k} \left[(-1)^i \eval{B}{\face{0}{k+1}}_{0,k+1}\right]+\\
    (-1)^{k+1} \eval{B}{\simplex{k}}_{0,k}\, U_{k,(k+1)}\, .
  \end{multline}
\end{proposition}
\begin{proof}
    The proof is a straightforward application of Definition~\ref{def:dnabla-hom} and is left to the reader.
\end{proof}

\begin{remark} 
    The above formula is entirely analogous to the discrete $d_\nabla$ applied to vector bundle-valued $k$-cochains, except for the modification in the last term. The necessity of this modification comes from tracking the source and target of the homomorphism-valued cochain: $\eval{B}{\simplex[1]{k+1}}_{1,k+1}$ and  $\eval{B}{\face{0}{k+1}}_{0,k+1}$ are maps from $E_{k+1}$ whereas $\eval{B}{\simplex{k}}_{0,k}$ is a map from~$E_k$. Thus, a transport from $k+1$ to $k$ is needed in the last term in~\eqref{eq:dnabla-hom}. Definition~\ref{def:dnabla-hom} correctly produces this modification. It is reassuring but not surprising that this required modification falls out of requiring that the Leibniz rule~\eqref{eq:Hom-Leibniz} hold.
\end{remark}

In the smooth setting, applying the exterior covariant derivative twice is the same as applying the operator corresponding to the curvature. The analogous result holds in the discrete setting.

\begin{proposition}\label{prop:dnabla2-alpha}
For a vector bundle-valued discrete $k$-cochain $\alpha \in C^k(X; E)$, we have
  \begin{equation} \label{eq:dnabla2isF}
  d_\nabla(d_\nabla(\alpha))=F_\nabla \smile \alpha.   
  \end{equation}
\end{proposition}
\begin{proof}
The statement is equivalent to the equality of values on $(k+2)$-simplices
\[
\eval{d_\nabla d_\nabla \alpha}{\simplex{k+2}}_0  = \eval{F \smile \alpha}{\simplex{k+2}}_0\, .
\]
We start by unpacking the left hand side as 
  \begin{equation} \label{eq:dnabla2LHS}
    \eval{d_\nabla d_\nabla \alpha}{\simplex{k+2}}_0
    = U_{01} \eval{d_\nabla \alpha}{\simplex[1,2]{k+2}}_1 +
    \sum_{i=1}^{k+2}(-1)^i \eval{d_\nabla \alpha}{\face{0,1}{k+2}}_0\, 
  \end{equation}
  and further expanding the first term on the right side of~\eqref{eq:dnabla2LHS},
  \begin{equation}\label{eq:simple-term}
    U_{01}\left(U_{12}\eval{\alpha}{\simplex[2]{k+2}}_2 +
      \sum_{i=2}^{k+2}(-1)^{i-1}\eval{\alpha}{\face{1,2}{k+2}}_1\right)\, .
  \end{equation}
  Separating the first term from the summation term in right side of~\eqref{eq:dnabla2LHS} we have
  \begin{equation} \label{eq:summation}
    (-1) \eval{d_\nabla \alpha}{\simplex[0,2]{k+2}}_0 +
    \sum_{i=2}^{k+2}(-1)^i \eval{d_\nabla \alpha}{\face{0,1}{k+2}}_0\, 
  \end{equation}
  and then using the definition of $d_\nabla$ for the first term in~\eqref{eq:summation} we have
  \begin{equation}\label{eq:U02term}
    (-1) U_{02}\eval{\alpha}{\simplex[2]{k+2}}_2 -
    \sum_{i=2}^{k+2}(-1)^{i-1}\eval{\alpha}{\face{0,2}{k+1}}_0\, .
  \end{equation}
  Expanding $d_\nabla$ in the second term in~\eqref{eq:summation} yields
  \begin{multline}\label{eq:summation-expanded}
    U_{01}\sum_{i=2}^{k+2}(-1)^i \eval{\alpha}{\face{1,2}{k+2}}_1 +\\
    \sum_{i=2}^{k+2}\left(
      \sum_{j=1}^{i-1}(-1)^{i+j}  
      \eval{\alpha}{\face[\hat{j} ... \hat{i}]{0,1,2}{k+2}}_0\right.
    \left.+\sum_{j=i+1}^{k+2}(-1)^{i+j-1}
      \eval{\alpha}{\face[\hat{i}...\hat{j}]{0,1,2}{k+2}}_0\right)\, .
  \end{multline}
  Thus the LHS of~\eqref{eq:dnabla2isF} is the sum of~\eqref{eq:simple-term},~\eqref{eq:U02term} and~\eqref{eq:summation-expanded}. The first term in~\eqref{eq:simple-term} and the first term in~\eqref{eq:U02term} combine to give the curvature $F$:
  \begin{align*}
   U_{01}U_{12}\eval{\alpha}{\simplex[2]{k+2}}_2 - U_{02}\eval{\alpha}{\simplex[2]{k+2}}_2 &= \eval{F}{[0,1,2]}_{0,2}\; \eval{\alpha}{\simplex[2]{k+2}}_2 \\
   &= \eval{F\smile\alpha}{\simplex[0]{k+2}}_0\;.   
  \end{align*}
The summation term of~\eqref{eq:simple-term} cancels the first summation term of~\eqref{eq:summation-expanded}. The terms that remain unaccounted for are the summation term in~\eqref{eq:U02term} and the double summation terms in~\eqref{eq:summation-expanded}. The $j=1$ term in the first double sum in~\eqref{eq:summation-expanded} is
  $\sum_{i=2}^{k+2} (-1)^{i+1}\eval{\alpha}{\face{0,2}{k+2}}_0$ which cancels with the second term in~\eqref{eq:U02term}. A simple re-indexing argument verifies that
  \begin{equation}\nonumber
    \sum_{i=2}^{k+2}\left(
      \sum_{j=2 }^{i-1}(-1)^{i+j}  
      \eval{\alpha}{\face[\hat{j} ... \hat{i}]{0,1,2}{k+2}}_0\right.
    \left.+\sum_{j=i+1}^{k+2}(-1)^{i+j-1}
      \eval{\alpha}{\face[\hat{i}...\hat{j}]{0,1,2}{k+2}}_0\right) = 0\, .
  \end{equation}
  The result is proved. 
  \end{proof}

\begin{proposition}[Discrete Bianchi identity] \label{prop:bianchi}
  The discrete curvature satisfies the Bianchi identity
  \[
  d_{\nabla^\Hom}\, F = 0\, .
  \]
\end{proposition}
\begin{proof}
  Consider a tetrahedron $[0,1,2,3]$. Then
  \begin{align*}
    \eval{d_{\nabla^\Hom}\, F}{[0,1,2,3]}_{0,3} &= U_{01} F_{123} - F_{023} +  F_{013}- F_{012} U_{23}\\
    &= U_{01}(U_{12} U_{23}- U_{13})  -(U_{02}U_{23}- U_{03}) + 
        (U_{01} U_{13}- U_{03}) - (U_{01} U_{12}- U_{02}) U_{23}\\
        &= 0,
  \end{align*}
  and the result follows. 
\end{proof}

\begin{remark}
  Similar combinatorial expressions of the Bianchi identity have been observed perviously, e.g.,~\cite{Kock1996,MiThWh1973,ChHu2023}. All of these formulas can be understood the combinatorics of curvature integrated over a 3-dimensional region together with the relationship between infinitesimal holonomy and curvature~\eqref{eq:Taylor}, e.g., see \cite[page 256]{BaMu1994}.
\end{remark}

\subsection{Curvature obstruction to anti-symmetrization of cup product}

\label{sec:crvtr_obstrctn}
The wedge product in DEC is defined as an anti-symmetrized cup product~\cite{Hirani2003}. Anti-symmetrization of the cup product for vector bundle-valued cochains (see~\S\ref{subsec:cup}) produces curvature terms that would break the Leibniz rule for the exterior covariant derivative, as we now explain. 

Let $X$ be the simplicial complex of tetrahedron $[0,1,2,3]$, $\alpha \in C^1(X; E)$ and $w \in C^1(X)$. Defining the wedge product $\alpha \wedge w$ as the anti-symmetrization of the cup product produces a sum over terms of the form $\sgn(\pi)\, \eval{d_\nabla(\alpha \smile w)}{[\pi(0),\pi(1),\pi(2),\pi(3)]}$ for all $\pi \in S_4$. For half of these terms a curvature obstruction appears that prevents Leibniz rule for being satisfied for that permutation. (In the following computation we use the shorthand $\alpha_{ij} = \eval{\alpha}{[i,j]}_i$ and $w_{ij} = \eval{w}{[i,j]}$.) For example, consider the permutation $(0,1,2,3)\mapsto (0,2,3,1)$. Then
  \begin{equation}\label{eq:crvtranti}
    \eval{d_\nabla(\alpha \smile w)}{[0,2,3,1]}_0 = \eval{d_\nabla\alpha \smile w}{[0,2,3,1]}_0 -
    \eval{\alpha \smile dw}{[0,2,3,1]}_0\, ,
  \end{equation}
  if and only if $U_{01}U_{12} = U_{02}$ which is equivalent to vanishing curvature in $[0,1,2]$. To see this note that LHS of\eqref{eq:crvtranti} is
  \begin{align}\label{eq:crvtrantiLHS}
    U_{01}\eval{\alpha\smile w}{[2,3,1]}_1 &- \eval{\alpha\smile w}{[0,3,1]}_0 +\eval{\alpha\smile w}{[0,2,1]}_0 - \eval{\alpha\smile w}{[0,2,3]}_0 = \\
    &U_{01}U_{12} \alpha_{23} w_{31} - \alpha_{03}w_{31} + \alpha_{02}w_{21} - \alpha_{02}w_{23}\, ,\notag
\end{align}
and the RHS of~\eqref{eq:crvtranti} is
\begin{align}\label{eq:crvtrantiRHS}
  \eval{d_\nabla\alpha}{[0,2,3]}_0 w_{31} &- \alpha_{02}\eval{dw}{[2,3,1]} =\\
  &U_{02}\alpha_{23}w_{31} - \alpha_{03}w_{31} + \alpha_{02}w_{31}-\alpha_{02}w_{31} + \alpha_{02}w_{21} - \alpha_{02}w_{23}\, .\notag
\end{align}
The RHS of~\eqref{eq:crvtrantiLHS} and~\eqref{eq:crvtrantiRHS} are equal iff $U_{01}U_{12}\alpha_{23}w_{31} = U_{02}\alpha_{23} w_{31}$.

A straightforward but tedious check shows that when the terms corresponding to the evaluations on other permutations of $(0,1,2,3)$ are considered together with signs there is no global cancellation in the antisymmetrization that would yield a Leibniz rule free of curvature obstruction. Thus in order for the Leibniz rule to hold, we are forced to use the cup product in Definition~\ref{def:cup} instead of a discrete wedge product.

\subsection{Connection 1-cochains and gauge transformations}\label{sec:connection-1form}
Following~\eqref{eq:smoothA}, in the smooth setting connections can be described (locally) by endomorphism-valued 1-forms. An analogous statement holds in the discrete setting. 

\begin{lemma}
Let $\nabla$ and $\nabla'$ be two (different) discrete connections on a discrete vector bundle $E\to X$. Then the difference
\[
\nabla-\nabla'\in C^1(X;E)
\]
defines a homomorphism-valued 1-cochain.
\end{lemma}

\begin{proof}
For $s\in C^0(X;E)$, we compute the evaluation on an edge as
\[
\Eval{(\nabla-\nabla')(s)}{[0,1]}_0 = (U_{01}-U'_{01}) s_1\, .
\]
We conclude $\nabla-\nabla' = U_{01} - U'_{01} \in C^1(X;\Hom(E))$
\end{proof}

For a discrete vector bundle $E$, denote $d_E$ denote the \emph{trivial connection}, meaning the parallel transport matrices defining $d_E$ are all identity maps. 

\begin{definition}\label{def:connection-1form}
For $(E,\nabla)$ a discrete vector bundle with connection, the \emph{discrete connection 1-cochain} is $A:=\nabla-d_E \in C^1(X;\Hom(E))$. 
We use the notation $A_{ij}:= \eval{A}{[i,j]}_{i,j} = U_{ij} - \Id_{ij}$ to denote the value of $A$ on an edge.
\end{definition}

The discrete analog of the expression for curvature in terms of the connection 1-form (see~\eqref{eq:smoothA}) requires the following product of homomorphism-valued forms.

\begin{definition}\label{def:B-cup-D}
    Given $B\in C^k(X;\Hom(E))$ and $D \in C^l(X;\Hom(E))$ the cup product of the two is $B\smile D \in C^{k+l}(X;\Hom(E))$ defined by its evaluation on $\simplex{k+l}$ as
    \begin{equation}\label{eq:b-cup-D}
    \eval{B\smile D}{\simplex{k+1}}_{0,k+1}:= \eval{B}{\simplex{k}}_{0,k}\;\eval{D}{\simplex[k]{k+l}}_{k,k+l}\, 
    \end{equation}
where the right hand side is a composition of linear maps. 
\end{definition}

\begin{lemma}\label{lem:F-using-A}
Given $(E,\nabla) \to X$ a discrete vector bundles with connection, discrete curvature $F_\nabla\in C^2(X;\Hom(E))$, and discrete connection 1-cochain $A \in C^1(X;\Hom(E))$, we have the equality 
\begin{equation}\label{eq:F-using-A}
    F_\nabla = d_E A + A\smile A\in C^2(X;\Hom(E))\, ,
\end{equation}
of homomorphism-valued 2-cochains. 
\end{lemma}
\begin{proof}
As usual, we verify~\eqref{eq:F-using-A} by evaluation on 2-simplices. Using~\eqref{eq:dnabla-hom},
\begin{align*}
    \eval{d_EA}{[0,1,2]}_{0,2} &= \Id_{01}\eval{A}{[1,2]}_{1,2} - \eval{A}{[0,2]}_{0,2} + \eval{A}{[0,1]}_{0,1}\Id_{12}\\
    &= \Id_{01}\, A_{12} - A_{02} + A_{01}\, \Id_{12}\, .
\end{align*}
The right hand side of~\eqref{eq:F-using-A} then evaluates to 
\begin{align*}
    \eval{d_EA}{[0,1,2]}_{0,2} + \eval{A}{[0,1]}_{0,1}\,\eval{A}{[1,2]}_{1,2}
    &=\Id_{01}\, A_{12} - A_{02} + A_{01}\, \Id_{12} + A_{01}\, A_{12}\, .
\end{align*}
Substituting $A_{ij}=U_{ij}-\Id_{ij}$ above, we obtain 
\begin{align*}
  &\Id_{01}(U_{12}-\Id_{12}) - (U_{02}-\Id_{02}) +(U_{01}-\Id_{01})\Id_{12} +     (U_{01}-\Id_{01})(U_{12}-\Id_{12})\\
  &=\Id_{01}U_{12}-\Id_{01}\Id_{12}-U_{02}+\Id_{02} +U_{01}\Id_{12} - \Id_{01}\Id_{12} + U_{01}U_{12} - \Id_{01}U_{12} - U_{01}\Id_{12} + \Id_{01}\Id_{12}\, .
\end{align*}
Then since $\Id_{01}\Id_{12} = \Id_{02}$ all terms above cancel except $U_{01}U_{02}-U_{02}$. But this is precisely $\eval{F}{[0,1,2]}_{0,2}$, the left side of~\eqref{eq:F-using-A}.
\end{proof}

The gauge transformations defined in \S\ref{sec:dscrtvctrbndls} can be recast as homomorphism-valued 0-cochains, $g\in C^0(X;\Hom(E))$. 
The classical formula~\eqref{eq:gaugetransf} for the image of a connection 1-form under a gauge transformation has the following discrete analog. 

\begin{proposition}
Given a gauge transformation $g \in C^0(X; \Hom(E))$ between discrete bundles with connection $(E,\nabla)\mapsto (E,\nabla')$, the corresponding connection 1-forms and curvature are related by 
\[
A\mapsto A'=g\,A\,g^{-1} - dg\, g^{-1},\quad F'=gFg^{-1}. 
\]
\end{proposition}
\begin{proof}
    Starting with the connection 1-forms, we evaluate both sides on edge 
    \begin{align*}
        \Eval{gAg^{-1}-dgg^{-1}}{[0,1]}&=\eval{g \smile A \smile g^{-1}}{[0,1]}_{0,1} - 
        \eval{dg\smile g^{-1}}{[0,1]}_{0,1} \\&= 
        g_0A_{01}g_1^{-1} - (\Id_{01}g_1 - g_0\Id_{01})g_1^{-1}\\
        &=g_0A_{01}g_1^{-1} - \Id_{01} + g_0\Id_{01}g_1^{-1}\\ 
        &=g_0U_{01}g_1^{-1} - g_0\Id_{01}g_1^{-1} - \Id_{01} + g_0\Id_{01}g_1^{-1}\, \\
        &=\eval{A'}{[0,1]} 
    \end{align*}
as desired. The expression for curvature follows from the above, or can be computed directy by the values $\eval{F}{[0,1,2]}_{0,2} = F_{012} \mapsto g_0 F_{012}g_2^{-1}$
\end{proof}

\section{Flat bundles and cohomology with local coefficients}\label{sec:flat}

\begin{definition}
A discrete vector bundle with connection is \emph{flat} if its curvature vanishes. 
\end{definition}

In view of Definition~\ref{def:discrete-curvature}, the flatness condition turns $(C^\bullet(X;E),d_\nabla)$ into a cochain complex. Below we identify the cohomology of this complex with cohomology in a certain local coefficient system on~$X$ built from a discrete vector bundle and (flat) discrete connection. 

\subsection{Local coefficients from flat bundles}
A \emph{local coefficient system} is a \emph{locally constant sheaf}, i.e., a sheaf that once restricted to sufficiently small open sets is isomorphic to a constant sheaf. All sheaves below are valued in real vector spaces, or equivalently, local coefficients are~$\R$-modules. 

\begin{lemma}\label{lem:construction}
On a space $X$, a locally constant sheaf $\mathcal{E}$ can be specified (up to canonical isomorphism) by the following data: (i) a good open cover $\{\mathcal{U}_i\}$ of $X$, (ii) vector spaces $\mathcal{E}_i$ for each open $\mathcal{U}_i\subset X$, and (iii) isomorphisms of vector spaces $\varphi_{ji}\colon \mathcal{E}_i\to \mathcal{E}_j$ for each nonempty intersection $\mathcal{U}_i\bigcap \mathcal{U}_j$. 
These satisfy the cocycle condition $\varphi_{kj}\circ \varphi_{ji} =\varphi_{ki}$ for $\mathcal{U}_i\bigcap \mathcal{U}_j\bigcap \mathcal{U}_k$ a nonempty triple intersection. 
\end{lemma}

\begin{proof}[Proof sketch.]
Choose an ordering on the index set of the cover. On a nonempty intersection, assign
\[
\mathcal{E}(\mathcal{U}_{i_0}\bigcap \cdots \bigcap\mathcal{U}_{i_k}):=\mathcal{E}_{i_0}
\]
where the indices are listed in increasing order, so $i_0$ is the lowest index. For a nonempty inclusion $\mathcal{U}_{i_0}\bigcap \cdots \bigcap\mathcal{U}_{i_k}\hookrightarrow \mathcal{U}_{j_0}\bigcap \cdots \bigcap \mathcal{U}_{j_l}$, define the restriction map from the given data as 
\begin{equation}\label{eq:restrictionmaps}
\mathcal{E}(\mathcal{U}_{j_0}\bigcap \cdots \bigcap\mathcal{U}_{j_l})=\mathcal{E}_{j_0}\xrightarrow{\varphi_{i_0j_0}} \mathcal{E}_{i_0}=\mathcal{E}(\mathcal{U}_{i_0}\bigcap \cdots \bigcap\mathcal{U}_{i_k}).
\end{equation}
 The cocycle condition on this data guarantees that iterated restrictions agree with the evident composition of restrictions. This determines the values of a sheaf $\mathcal{E}$ on an open cover, with transition data on overlaps specified by isomorphisms. For any other open subset $\mathcal{U}\subset X$, the above values, local constancy, and the sheaf condition determine the value~$\mathcal{E}(\mathcal{U})$. A different choice of ordering on the index set of the cover determines a canonically isomorphic sheaf. 
\end{proof}

To apply this general sheaf theory to a discrete vector bundle, suppose we have a triangulation of a smooth manifold~$M$ as part of the data of the ordered simplicial complex $X$. Let $\mathcal{U}_i\subset M$ denote the open subset given by the (open) star of the vertex $i\in X_0$; the cover $\{\mathcal{U}_i\}$ is a good open cover (e.g., intersections are convex). Furthermore, $\mathcal{U}_i\bigcap \mathcal{U}_j\subset M$ is nonempty if and only if there is an edge $[i,j]\in X_1$ connecting the vertices. More generally, a $k$-fold intersection $\mathcal{U}_{i_1}\bigcap \cdots \bigcap \mathcal{U}_{i_k}$ is nonempty if and only if the participating vertices determine a $k$-simplex in $X_k$.

\begin{lemma}\label{construction:localsystem}
A discrete vector bundle with flat connection $(E,\nabla)\to X$ determines a locally constant sheaf $\mathcal{E}$ via (i) the good cover $\{\mathcal{U}_i\}$ given by open stars of vertices, (ii) the local values $\mathcal{E}_i:=E_i$, and (iii) the gluing isomorphisms $\varphi_{ji}=U_{ji}\colon E_i\to E_j$ determined by parallel transport maps. 
A map of discrete vector bundles determines a morphism of the corresponding sheaves.
\end{lemma}
\begin{proof}
Applying Lemma~\ref{lem:construction}, we need only check the cocycle condition $\varphi_{kj}\circ \varphi_{ji} =\varphi_{ki}$; but the vanishing of discrete curvature $F=d_\nabla^2=0$ is exactly the desired equality, $U_{kj} U_{ji} =U_{ki}$. A map of discrete vector bundles with connection is the data of maps of vector spaces $E_i\to E_i'$ compatible with the parallel transport maps; this is precisely the data of a morphism of locally constant sheaves compatible with restriction maps~\cite[page 109]{BoTu1982}. 
\end{proof}

\subsection{A discrete model for cohomology with local coefficients}

\begin{proposition}\label{prop:Cech}
Let $(E,\nabla)\to X$ be a flat discrete vector bundle with connection. The cochain complex $(C^\bullet(X;E),d_\nabla)$ equals the \v{C}ech complex for cohomology of the sheaf $\mathcal{E}$ for the good open cover $\{\mathcal{U}_i\}$ in the notation above. Furthermore, the module structure of \v{C}ech complex for the constant sheaf valued in $\R$ agrees with the cup product~\eqref{eq:homAction}. 
\end{proposition}

\begin{proof}
A standard reference for \v{C}ech cohomology valued in a locally constant sheaf is~\cite[\S10]{BoTu1982}. We compare the \v{C}ech complex with the discrete complex $(C^\bullet(X;E),d_\nabla)$ via the open cover given by stars of vertices. The $\mathbb{Z}$-graded vector space underlying the $\mathcal{E}$-twisted \v{C}ech complex has in degree $k$-elements the values of the sheaf on $k$-fold intersections $\mathcal{U}_{i_0}\bigcap \dots \bigcap \mathcal{U}_{i_k}$. Identifying a nonempty $k$-fold intersection with a $k$-simplex in $X$, this $\mathbb{Z}$-graded vector space is the same as the vector space underlying the $E$-twisted simplicial cohomology: the intersection of a $(k+1)$-tuple $\mathcal{U}_{i_0}\bigcap \cdots \mathcal{U}_{i_k}$ is non-empty if and only if the vertices $i_0,\dots, i_k$ span a $k$-simplex, and a $k$-cochain assigns an element of the vector space $E_{i_0}$ of the lowest vertex of this $k$-simplex. 

Next we compare the differentials on these $\mathbb{Z}$-graded vector spaces. The \v{C}ech differential is defined by the usual alternating sum of restrictions, \cite[page 110]{BoTu1982}. For the conventions on these restrictions as in~\eqref{eq:restrictionmaps}, $d_\nabla$ is precisely equal to the \v{C}ech differential under the identification in the previous paragraph: all but one of the restriction maps is the identity map, and the remaining restriction map involves a single parallel transport, just as in~\eqref{eq:dnabla}. 

The statement about cup products follows  from the definition of the product in \v{C}ech cohomology as the product of (locally constant) values, e.g.,~\cite[Equation~14.27]{BoTu1982}.
\end{proof}

Now let $E\to M$ be a smooth vector bundle with flat connection~$\nabla$. Define the locally constant sheaf $\mathcal{E}'$ on $M$ with values 
\begin{equation}\label{eq:smoothsheaf}
\mathcal{E}'(\mathcal{U}):=\{v\in \Gamma(\mathcal{U};E)\mid \nabla v=0\},\quad \mathcal{U}\subset M
\end{equation}
with restriction maps for $\mathcal{E}'$ from restricting (locally constant) sections. 

\begin{lemma}\label{lem:smoothcech}
Given a smooth vector bundle with flat connection, the discretization of $(E,\nabla)$ following \S\ref{sec:form-discretization} produces a sheaf under Proposition~\ref{prop:Cech} that is isomorphic to~\eqref{eq:smoothsheaf}. 
\end{lemma}
\begin{proof}
Proceeding with the same notation as in the proof of Proposition~\ref{prop:Cech}, to produce an isomorphism we need to verify that the values are isomorphic and the restriction maps agree:
$$
\mathcal{E}'(\mathcal{U}_i)\simeq \mathcal{E}(\mathcal{U}_i),\qquad U_{ji}={\rm res}_{ji}.
$$
We obtain an isomorphism from the fact that parallel transport along a flat connection depends only on the homotopy class of the path, and there is a unique such homotopy class for any pair of points in the (convex) open star $\mathcal{U}_i$; this identifies a covariantly constant section~\eqref{eq:smoothsheaf} with its value at $i\in \mathcal{U}_i$. The equality of restriction maps then comes from comparing the restrictions to different basepoints, which is exactly implemented by parallel transport along the edge connecting the points. 
\end{proof}

\begin{proof}[Proof of Theorem~\ref{thm:Cech}]
For a smooth manifold $M$ and a flat vector bundle $(E,\nabla)\to M$, the $E$-twisted de~Rham cohomology of $M$ agrees with \v{C}ech cohomology valued in the sheaf~\eqref{eq:smoothsheaf}, e.g., by the a \v{C}ech--de~Rham double complex as in~\cite{BoTu1982}. Hence, it suffices to compare the simplicial complex for a flat discrete connection with the  \v{C}ech complex valued in $\mathcal{E}'$. The claimed result then follows from Proposition~\ref{prop:Cech} and Lemma~\ref{lem:smoothcech}. 
\end{proof}

\begin{remark}
    The literature has other closely related versions of the twisted and simplicial de~Rham theorems, e.g., see~\cite[Chapter 14]{Halperin1983}. 
\end{remark}

\begin{example}[Orientations and twisted Poincar\'e duality]\label{ex:twistedPoincarediscrete}
A particularly important twisted de~Rham complex comes from twisting by the \emph{orientation line}, which is equivalently the top exterior power of the tangent bundle, $\Lambda^{{\rm dim}(TM)}(M)$. Explicitly, one fixes charts $\mathcal{U}_i$ determined by the triangulation, and defines a discrete line bundle with flat connection by the values $E_i=\R$ for all vertices $i\in X_0$, and the linear map $+1$ (or $-1$) according to whether the transition map on $\mathcal{U}_i\bigcap \mathcal{U}_j$ preserves (or reverse) the orientation of the tangent bundle of $M$. As a consequence of Theorem~\ref{thm:Cech}, we find that the associated complex $(C^\bullet(X;\Omega^{{\rm dim}(M)}(M)),d_\nabla)$ of discrete cochains computes the orientation-twisted de~Rham cohomology. By the \v{C}ech--de~Rham double complex isomorphism, the cohomology of this complex satisfies twisted Poincar\'e duality in the sense of~\eqref{eq:twistedPD}. It remains an interesting open question to realize this twisted Poincar\'e duality at the cochain level, much as the Hodge $*$-operator (and its discrete analog) realizes Poincar\'e duality for differential forms. 
\end{example}

\begin{example}
When applied to the constant sheaf valued in $\R$, Proposition~\ref{prop:Cech} gives a new argument demonstrating that DEC computes de~Rham cohomology. Indeed, we may identify discrete cochains with  \v{C}ech cochains valued in the constant sheaf valued in~$\R$, and then the \v{C}ech--de~Rham double complex gives an algebraic proof that DEC computes de~Rham cohomology of (triangulated) smooth manifolds. We note that there are no convergence issues to be verified in this approach, relegating all analysis to the \v{C}ech--de~Rham comparison~\cite{BoTu1982}.
\end{example}

\section{Coarsened theory}\label{sec:coarse}

Let $\X$ be a (possibly unordered) simplicial complex and $X := \sd \X$ the corresponding subdivided ordered one. As noted in~\S\ref{subsec:subdiv} we will order vertices of each simplex in $\sd \X$ by decreasing dimension and that abstract simplicial maps subdivide to yield order-preserving ones. In this setting we can define a coarsened version of our theory that applies to the coarser complex $\X$. As noted in Remark~\ref{rem:subdiv-notation} when $\X$ is a geometric simplicial complex interpret $c(\tsigma)$ for $\tsigma \in \X$  as a point associated with $\tsigma$, typically in its interior. For example this can be the barycenter. With future possibilities in mind, this point could be a circumcenter as in DEC. We will call this distinguished point $c(\tsigma)$ the \emph{center} of the simplex. The coarsening procedure in this section reproduces the theory of discrete vector bundles with connections of Christiansen and Hu~\cite{ChHu2023} (CH-bundles). In addition approaching their theory from this direction reveals possible difficulties in defining products which have not been resolved yet.

\begin{definition}
    For $X:=\sd \X$ and $(E,\nabla) \to X$ the data for the \emph{coarsened} vector bundle with connection $(\E,\tnabla) \to \X$ has fibers $\E_{c(\tsigma)}$ for $\tsigma\in \X$ and parallel transport maps $U_{c(\tsigma^k),c(\tsigma^{k+1})}$ for all $\tsigma^k\prec\tsigma^{k+1} \in \X$.
\end{definition}

\begin{remark}
By construction, fibers of $\E$ and $E$ coincide and transport maps of $\tnabla$ are a subset of the maps of $\nabla$. The definition of cochains in the coarse theory is the same as Definition~\ref{def:cochain} but the origin vertex of each $\tsigma\in \X$ is now $c(\tsigma)$. The space of $k$-cochains in the coarse theory will be denoted $\C^k(\X;\E)$. 
\end{remark}

\begin{definition}
    For $\alpha \in C^k(X;E)$ and $\tsigma^k \in \X$ the \emph{cochain coarsening map} $\coarse: C^k(X;E) \to \C^k(\X;\E)$ is defined by 
    \[ 
    \eval{\coarse(\alpha)}{\tsigma^k}_{c(\tsigma^k)} = 
    \sum_{\sigma^k\in \sd \tsigma^k\subset X} o(\sigma^k,\tsigma^k)\;\eval{\alpha}{\sigma^k}_{c(\tsigma^k)}\, .
    \]
\end{definition}

\begin{remark}
    The lowest vertex convention of fine theory and the choice of ordering vertices in $\sd \X$ by reverse dimension are designed to achieve a very important goal. This is that $\eval{\alpha}{\sigma^k}\in E_{c(\tsigma^k)} = \E_{c(\tsigma^k)}$ for all $\sigma^k\in \sd \tsigma^k$ since $c(\tsigma^k)$ is the smallest element in the vertex set of each $\sigma^k \in \sd \tsigma^k$ (Remark~\ref{rem:subdiv-vertex-order}). Thus the coarsening procedure is a simple addition with orientation without any transports needed.
\end{remark}

\begin{definition}
    The \emph{coarse connection} $\tnabla:\C^0(\X;\E)\to \C^1(\X;\E)$ is defined by its value on edge $\tsigma^1 = [\tilde{u},\tilde{v}] \in \X$ as
    \[
    \eval{\tnabla \tilde{s}}{\tsigma^1}_{c(\tsigma^1)}  := U_{c(\tsigma^1),\,c(\tilde{v})} \eval{\tilde{s}}{\tilde{v}}- U_{c(\tsigma^1),\,c(\tilde{u})}\eval{\tilde{s}}{\tilde{u}}\, ,
    \]
    for $\tilde{s} \in \C^0(\X;\E)$. The \emph{coarse discrete exterior covariant derivative}
    $d_\tnabla:\C^k(\X;\E)\to \C^{k+1}(\X;\E)$ is defined by its value on $\tsigma^{k+1} \in \X$ as
    \[
    \eval{d_\tnabla\tilde{\alpha}}{\tsigma^{k+1}}_{c(\tsigma^{k+1})} = 
    \sum_{\tsigma^k\prec\tsigma^{k+1}\in \X} o(\tsigma^k,\tsigma^{k+1})\;
    U_{c(\tsigma^{k+1}),c(\tsigma^k)}\;
    \eval{\tilde{\alpha}}{\tsigma^k}_{c(\tsigma^k)}\, ,
    \]
    for $\tilde{\alpha}\in\C^k(\X;\E)$.
\end{definition}

\begin{proposition}
    The coarsening map commutes with the exterior covariant derivative, $d_\tnabla \coarse = \coarse d_\nabla$.
\end{proposition}
\begin{proof}
    Let $\alpha \in C^k(X;E)$ and $\tsigma^{k+1}\in \X$. Then
    \begin{align}\label{eq:coarse-commute-fullsum}
        \Eval{\coarse(d_\nabla\alpha)}{\tsigma^{k+1}}_{c(\tsigma^{k+1})} &=
        \sum_{\sigma^{k+1}\in \sd \tsigma^{k+1}} o(\sigma^{k+1},\tsigma^{k+1})\;
        \sum_{\tau^k \prec \sigma^{k+1}} o(\tau^k, \sigma^{k+1})\; 
        U_{c(\tsigma^{k+1}),v(\tau^k)}\;
        \eval{\alpha}{\tau^k}_{v(\tau^k)}\, ,
    \end{align}
    where $v(\tau^k)$ is the origin vertex of $\tau^k$. The set of all $\tau^k \prec \sigma^{k+1} \in \sd\tsigma^{k+1}$ is the union of two disjoint sets of $\tau^k$, those that are in the boundary of $\tsigma^{k+1}$ and those that are not, or more precisely those that are in $\sd \partial \tsigma^{k+1}$ and the others. 
    
    For $\tau^k\in \sd \partial \tsigma^{k+1}, \exists!\, \sigma^{k+1} \succ \tau^k$ with $\sigma^{k+1}\in \sd \tsigma^{k+1}$ and if $\tau^k \in \sd \tsigma^k, \tsigma^k \prec \tsigma^{k+1}$ the origin vertex is $v(\tau^k) = c(\tsigma^k)$ since that is the lowest vertex for any such $\tau^k \prec \tsigma^k$.
    For $\tau^k \not\in \sd\partial \tsigma^{k+1}, \exists! (\sigma^{k+1}_+, \sigma^{k+1}_-), \sigma^{k+1}_\pm \succ \tau^k, \sigma^{k+1}_\pm \in \sd \tsigma^{k+1}$ and the origin vertex $v(\tau^k) = c(\tsigma^{k+1})$ since that is the lowest vertex in $\sd \tsigma^{k+1}$. Thus the RHS of~\eqref{eq:coarse-commute-fullsum} can be written as
    \begin{multline}\label{eq:coarse-commute-split}
        \sum_{\tsigma^k\prec\tsigma^{k+1}}\sum_{\tau^k\in \sd \tsigma^k}
        o(\sigma^{k+1},\tsigma^{k+1})\; o(\tau^k,\sigma^{k+1})\;
        U_{c(\tsigma^{k+1}),c(\tsigma^k)}\;
        \eval{\alpha}{\tau^k}_{c(\tsigma^k)} +\\
        \sum_{\tau^k\not\in \sd \partial\tsigma^{k+1}} 
        o(\sigma^{k+1}_+,\tsigma^{k+1})\; o(\tau^k,\sigma^{k+1}_+)\;
        \eval{\alpha}{\tau^k}_{c(\tsigma^{k+1})} +
        o(\sigma^{k+1}_-,\tsigma^{k+1})\; o(\tau^k,\sigma^{k+1}_-)\;
        \eval{\alpha}{\tau^k}_{c(\tsigma^{k+1})}\,.
    \end{multline}
    The transport operators needed in the second summation are $U_{c(\tsigma^{k+1}),c(\tsigma^{k+1})}$ which are identity and thus not written. For $\tsigma^k \prec \tsigma^{k+1}, \tau^k \in \sd \tsigma^k$ the orientation factor can be rewritten using
    \[
    o(\sigma^{k+1},\tsigma^{k+1})\; o(\tau^k,\sigma^{k+1}) = 
    o(\tau^k,\tsigma^k)\; o(\tsigma^k,\tsigma^{k+1})\, ,
    \] 
    where $\sigma^{k+1}$ is the unique $(k+1)$-simplex in $\sd \tsigma^{k+1}$ containing $\tau^k$. For $\tau^k \not\in \sd\partial\tsigma^{k+1}$ the orientation factor satisfies 
    \[
    o(\sigma^{k+1}_+,\tsigma^{k+1})\; o(\tau^k,\sigma^{k+1}_+) = -o(\sigma^{k+1}_-,\tsigma^{k+1})\; o(\tau^k,\sigma^{k+1}_-)\, ,
    \]
    where $\sigma^{k+1}_\pm$ are the unique $(k+1)$-simplices in $\sd \tsigma^{k+1}$ containing $\tau^k$. Thus~\eqref{eq:coarse-commute-split} simplifies to give
    \begin{align*}
        \Eval{\coarse(d_\nabla\alpha)}{\tsigma^{k+1}}_{c(\tsigma^{k+1})} &=
        \sum_{\tsigma^k\prec\tsigma^{k+1}}
        \sum_{\tau^k\in \sd \tsigma^k}
        o(\tau^k,\tsigma^k)\; o(\tsigma^k,\tsigma^{k+1})\;
        U_{c(\tsigma^{k+1}),c(\tsigma^k)}\;
        \eval{\alpha}{\tau^k}_{c(\tsigma^k)}  + 0\\
        &= \sum_{\tsigma^k\prec\tsigma^{k+1}} o(\tsigma^k,\tsigma^{k+1})\;
        U_{c(\tsigma^{k+1}),c(\tsigma^k)}\;
        \sum_{\tau^k\in \sd \tsigma^k} o(\tau^k,\tsigma^k)\;
        \eval{\alpha}{\tau^k}_{c(\tsigma^k)}\\
        &= \sum_{\tsigma^k\prec\tsigma^{k+1}} o(\tsigma^k,\tsigma^{k+1})\;
        U_{c(\tsigma^{k+1}),c(\tsigma^k)}\;
        \Eval{\coarse(\alpha)}{\tsigma^k}_{c(\tsigma^k)}\\
        &= \Eval{d_\tnabla(\coarse(\alpha))}{\tsigma^{k+1}}_{c(\tsigma^{k+1})}\, .
    \end{align*}
\end{proof}

\begin{definition}
   For $\tsigma^k,\tsigma^{k+l} \in \X$, $0 \le k \le \dim \X - l$, a \emph{coarse homomorphism valued $l$-cochain} assigns to each $l$-dimensional combinatorial cube $D(\tsigma^k\preceq\tsigma^{k+l}) \in \Q(\X)$ a homomorphism $\E_{c(\tsigma^k)} \to \E_{c(\tsigma^{k+l})}$. Here $\Q(\X)$ is the cubical refinement of $\X$. The spaces of homomorphism valued $l$-cochains will be denoted $\C^l(\Q(\X);\Hom(\E))$ and the evaluation of $\tilde{B} \in \C^l(\Q(\X);\Hom(\E))$ on the $l$-dimensional cube $D(\tsigma^k\preceq\tsigma^{k+l}) \in \Q(\X)$ will be denoted $\Eval{\tilde{B}}{D(\tsigma^k\preceq\tsigma^{k+l})}_{c(\tsigma^{k+l}), c(\tsigma^k)}$ with subscripts indicating that this is a homomorphism $\E_{c(\tsigma^k)} \to \E_{c(\tsigma^{k+l})}$.
\end{definition}

\begin{remark}
    A remark similar to Remark~\ref{rem:why-no-homE} can be made here. We are using $\Hom(\E)$ without defining it. Given $\E \to \X$ we could have defined coarse homomorphism bundle $\Hom(\E)$ over $\Q(\X)$ to have the data of all linear maps $\E_{c(\tsigma^k)} \to \E_{c(\tsigma^{k+l})}$ for each cube $D(\tsigma^k\preceq\tsigma^{k+l})$. However, as noted in that earlier remark, spaces of cochains would then be sections of this bundle.
\end{remark}
    
\begin{remark} \label{rem:additive-Hom}
    Let $X := \sd \X$ and consider $E\to X$ and $\Hom(E) \to X$. The $l$-dimensional cube $D(\tsigma^k\preceq\tsigma^{k+l}) \in \Q(\X)$ for $\tsigma^k, \tsigma^{k+l} \in \X$ has the property that all the $l!$ simplices from $X$ of dimension $l$ that make it up have the same smallest vertex and largest vertex (Remark~\ref{rem:cube-vertex-order}). As a result the evaluations of a homomorphism valued cochain in $C^l(X; \Hom(E))$ in all these $l!$ simplices in $X$ can be added. This is useful in coarsening of homomorphism valued cochains defined next.
\end{remark}

\begin{definition}\label{def:hom-coarsening}
    For $B \in C^l(X;\Hom(E))$ and $\tsigma^k \preceq \tsigma^{k+l} \in \X$, the \emph{homomorphism coarsening map} $\coarse^\Hom: C^l(X;\Hom(E)) \to \C^l(Q(\X);\Hom(\E))$ is defined by 
    \begin{multline*}
      \Eval{\coarse^\Hom(B)}{D(\tsigma^k\preceq\tsigma^{k+l})}_{c(\tsigma^{k+l}),c(\tsigma^k)} =\\
    \sum_{c[\tsigma^{k+l}\succeq \tsigma^k] \in \sd \X}
    \sgn(c[\tsigma^{k+l}\succeq \tsigma^k])\;\Eval{B}{c[\tsigma^{k+l}\succeq \tsigma^k]}_{c(\tsigma^{k+l}),c(\tsigma^k)}\, ,        
    \end{multline*}

    where $\sgn(c[\tsigma^{k+l}\succeq \tsigma^k]) = (-1)^{kl}o(c[\tsigma^k\succeq\tsigma^0],\tsigma^k)\;
    o(c[\tsigma^{k+l}\succeq\tsigma^0],\tsigma^{k+l})$.
\end{definition}

\begin{definition}\label{def:coarse-curvature}
    Given $(\E,\tnabla) \to \X$, the \emph{coarse discrete curvature} of $\tnabla$ , denoted $F_\tnabla$, is an element of $C^2(Q(\X);\Hom(\E))$. For any pair of simplices $\tsigma^k \prec \tsigma^{k+2} \in \X, 0 \le k \le \dim \X - 2$, the evaluation $\Eval{F_\tnabla}{D(\tsigma^k\prec \tsigma^{k+2})}_{c(\tsigma^{k+2}),c(\tsigma^k)}$ is a homomorphism $\E_{c(\tsigma^k)} \to \E_{c(\tsigma^{k+2})}$. It is defined as
    \begin{multline*}
    \Eval{F_\tnabla}{D(\tsigma^k\prec\tsigma^{k+2})}_{c(\tsigma^{k+2}),c(\tsigma^k)} := \\
    \pm\big(
        U_{c(\tsigma^{k+2}),\,c(\tsigma_1^{k+1})}\;U_{c(\tsigma_1^{k+1}),\, c(\tsigma^k)} \; - 
        U_{c(\tsigma^{k+2}),\,c(\tsigma_2^{k+1})}\;U_{c(\tsigma_2^{k+1}),\, c(\tsigma^k)}
        \big)\, ,     
    \end{multline*}
    where $\tsigma^k \prec \tsigma^{k+1}_1, \tsigma^{k+1}_2 \prec \tsigma^{k+2}$.
    The $\pm$ sign on the RHS depends on the orientation of $D(\tsigma^k\prec\tsigma^{k+2})$, $\tsigma^k$ and $\tsigma^{k+2}$.
\end{definition}

\begin{proposition}
With $\X, X=\sd \X, (E,\nabla)\to X$ and $(\E,\tnabla)\to \X$ as above, coarse curvature $F_\tnabla$ is the coarsened fine curvature. That is, $F_\tnabla = \coarse^\Hom(F_\nabla)$.
\end{proposition}
\begin{proof}
    Follows from definitions.
\end{proof}

We will not formally define $d_{\tnabla^\Hom}$ but instead point to the corresponding definition in ~\cite{ChHu2023} and note that $\coarse\; d_{\nabla^\Hom} = d_{\tnabla^\Hom} \coarse^\Hom$.

\subsection{Difficulties in defining coarse products}
Defining a product in the coarse setting turns out to be challenging and is an unresolved issue. The coarsening of the finer cup does not lead to a cup in the coarse setting. That is, $\coarse(\alpha\smile w) \ne \coarse(\alpha) \smile \coarse(w)$ but instead needs values of fine $\alpha$. This can be seen in a simple computation using 1-cochains. Another alternative to consider is to use the ordinary cup product as we have used earlier. This however leads to a curvature obstruction in the Leibniz rule. While in the finer theory the curvature obstruction only appears if cup is anti-symmetrized, in the coarse theory the obstruction shows up even for the ordinary cup product as shown in the following example.

\begin{example}\label{ex:CHcrvtr_obstrctn}
  Let $(\E,\tnabla)\to \X$ be a coarse vector bundle with connection over $\X$, the simplicial complex of a tetrahedron $[0,1,2,3]$ with cochains $\alpha \in C^1(\X; \E)$ and $w \in C^1(\X)$. One would like that a Leibniz rule 
  \begin{equation} \label{eq:CHLeibniz}
  d_\tnabla(\alpha \smile w) = d_\tnabla\alpha \smile w - \alpha \smile dw
\end{equation}
hold 
when both sides are evaluated on $[0,1,2,3]$. We will use the shorthand $U_{0123,123}$ to mean transport from $c([1,2,3])$ to $c([0,1,2,3])$ etc. Using the definition of $d_\tnabla$ the value of the LHS $\eval{d_\tnabla(\alpha\smile w)}{[0,1,2,3]}$ is
\begin{multline*}
U_{0123,123} \eval{\alpha \smile w}{[1,2,3]} - U_{0123,023}\eval{\alpha\smile w}{[0,2,3]} \\
+ U_{0123,013}\eval{\alpha \smile w}{[0,1,3]} - U_{0123,012}\eval{\alpha\smile w}{[0,1,2]}\, ,
\end{multline*}
which finally evaluates to
\begin{align}\label{eq:CHLeibnizLHS}
  &U_{0123,123} U_{123,12}\alpha_{12}w_{23} - U_{0123,023} U_{023,02} \alpha_{02}w_{23} \notag\\
  &\qquad + U_{0123,013} U_{013,01} \alpha_{01}w_{13} - U_{0123,012} U_{012,01}\alpha_{01}w_{12}\, .
\end{align}
The first term on the RHS of~\eqref{eq:CHLeibniz} yields
\begin{align}\label{eq:CHLeibnizRHS1}
  \eval{d_\tnabla \alpha \smile w}{[0,1,2,3]} &= U_{0123,012}\eval{d_\tnabla \alpha}{[0,1,2]}w_{23}\notag\\
  &= U_{0123,012}U_{012,12}\alpha_{12}w_{23} - U_{0123,012}U_{012,02}\alpha_{02}w_{23} \notag\\
  &\qquad + U_{0123,012}U_{012,01}\alpha_{01}w_{23}\,
\end{align}
and the second term on the RHS of~\eqref{eq:CHLeibniz} yields
\begin{align} 
  -\eval{\alpha\smile dw}{[0,1,2,3]} &= -U_{0123,01}\alpha_{01}\eval{dw}{[1,2,3]} \notag \\
  &=-U_{0123,01}\alpha_{01}w_{23} + U_{0123,01}\alpha_{01}w_{13} \notag\\
  &\qquad - U_{0123,01}\alpha_{01}w_{12}\, .
\label{eq:CHLeibnizRHS2}  
\end{align}
The necessary and sufficient conditions for~\eqref{eq:CHLeibnizLHS} to equal the sum of \eqref{eq:CHLeibnizRHS1} and \eqref{eq:CHLeibnizRHS2} are obtained by matching terms and these conditions are
\begin{align*}
  U_{0123,123} U_{123,12} &= U_{0123,012} U_{012,12}\\
  U_{0123,023} U_{023,02} &= U_{0123,012} U_{012,02}\\
  U_{0123,013} U_{013,01} &= U_{0123,01}\\
  U_{0123,012} U_{012,01} &= U_{0123,01}\, .
\end{align*}
The last two can be combined to yield the following necessary conditions for Leibniz rule to hold in this case:
\begin{align*}
  U_{0123,123} U_{123,12} - U_{0123,012} U_{012,12} &= 0\\
  U_{0123,023} U_{023,02} - U_{0123,012} U_{012,02} &= 0\\
  U_{0123,013} U_{013,01} - U_{0123,012} U_{012,01} &= 0\, .
\end{align*}

These are the discrete curvatures associated with the edges of the triangle $[0,1,2]$ in the CH-bundles framework.
Thus for Leibniz rule~\eqref{eq:CHLeibniz} to be satisfied on the tetrahedron $[0,1,2,3]$ these curvatures must vanish.
\end{example}

\section{Conclusion and outlook}

We have given a combinatorial framework for studying discretizations of vector bundles with connection. Using the discrete covariant derivative $d_\nabla$ as the building block, curvature emerges natural as a homomorphism valued 2-cochain. Various structures in smooth geometry have immediate discrete counterparts, e.g., curvature as ``small" holonomy, the Bianchi identity for the curvature operator, connection 1-cochains, gauge transformations, and flat connections leading to cohomology twisted by local coefficients. From the start, we built in naturality for all of these discrete geometric objects with respect to (locally ordered) simplicial maps. This naturality is both an important theoretical constraint and also opens a door to potential applications in (discrete) dynamics. 

While there is no ``unique" or ``best" framework for discrete vector bundles, the framework developed in this paper seems to be quite flexible, allowing for comparisons with work of other authors. In particular, we provided a coarsening procedure that forges an explicit link with the work of Christiansen and Hu~\cite{ChHu2023}. 

The work of~\cite{BrToGaDe2024} builds on an earlier version of our paper. Alternation procedures pursued in their setup provide certain convergence results, though may also have the effect of obscuring certain geometric structures like the product of a homomorphism-valued form with a vector-valued form. Ultimately, one will almost surely need to tune the theoretical framework to the desired application. 

Such applications largely center around partial differential equations for vector valued-forms, which typically incorporate metric geometry via the Hodge star operator. This is perhaps most famous in the (vacuum) Yang--Mills equation $d_\nabla \ast F_\nabla = 0$. Scalar-valued DEC provides some hints for how to build such a Hodge star operator on vector bundle-valued cochains, but many questions remain. This is an important future direction. 

Once these classical PDE have discrete counterparts, one can ask whether discrete solutions converge to smooth solutions. Such questions can be quite subtle, depending sensitively on the regularity of solutions sought and the norm in which convergence is measured. This is another important long-term goal of the area. 

A purely cochains version of the Bernstein-Gelfand-Gelfand construction~\cite{CaSlSo2001} as used in numerical analysis~\cite{ArHu2021} is another important goal. In this setting, the discretization of double forms as cochains and first Bianchi sum operators for cochains are important future directions on which partial progress has been made~\cite{BeGa2025,HuLi2025,Zhu2026}.

\bigskip
\noindent\textbf{Acknowledgement:} ANH was supported in part by NSF DMS-2208581 and by KAUST Office of Sponsored Research under Award URF/1/3723-01-01. DBE was supported in part by NSF grants DMS-2205835 and DMS-2340239. We thank Georg Sprenger comments on an earlier version of this paper and Snorre Christiansen for discussions and comments on the current and previous versions. We thank Evan Gawlik for discussions on convergence issues.

\printbibliography

\end{document}